\documentclass[10pt]{article}
\usepackage{array,amsxtra}
\usepackage{amsmath,latexsym,amssymb}
\usepackage{amsfonts}
\usepackage{amsthm}
\usepackage{appendix}
\usepackage{mathtools}

\usepackage{pgfplots,pgfplotstable,enumerate,epsfig,graphicx,setspace,graphics,color,geometry,subfigure}                
\usepackage{soul}
\usepackage{fp}
\usepackage{tikz}
\usepackage{xcolor}
\usetikzlibrary{arrows}
\usetikzlibrary{shapes}
\usetikzlibrary{decorations.pathmorphing}
\usetikzlibrary{decorations.pathreplacing}
\usetikzlibrary{decorations.shapes}
\usetikzlibrary{decorations.text}

\newtheorem{theorem}{Theorem}[section]
\newtheorem{proposition}[theorem]{Proposition}

\newtheorem{definition}[theorem]{Definition}
\newtheorem{lemma}[theorem]{Lemma}
\newtheorem{corollary}[theorem]{Corollary}
\newtheorem{remark}[theorem]{Remark}

\newtheorem{conjecture*}{Conjecture}

\newcommand{\wX}{{\Xi}_{\mathrm{ext}}}
\newcommand{\X}{\widetilde{\Xi}}
\newcommand{\Xis}{\Xi}

\newcommand{\diam}{\mathrm{diam}}
\newcommand{\wch}{\widetilde{\chi}}
\newcommand{\vol}{\mathrm{vol}}
\newcommand{\A}{\mathcal{A}}

\newcommand{\spam}{\mathop{\mathrm{span}}}

\newcommand{\supp}[1]{{\text{supp}}({#1})}
\newcommand{\ints}{\mathbb{Z}}
\newcommand{\nats}{\mathbb{N}}
\newcommand{\N}{\mathbb{N}}
\newcommand{\reals}{\mathbb{R}}
\newcommand{\RR}{\mathbb{R}}
\renewcommand{\S}{\mathbb{S}}
\newcommand{\comps}{\mathbb{C}}
\newcommand{\M}{\Omega}

\newcommand{\dif}{\mathrm{d}}

\renewcommand{\a}{\mathbf{a}}

\newcommand{\dist}{\mathrm{dist}}

\renewcommand{\d}{\mathrm{dist}}
\renewcommand{\A}{\mathcal{A}}%annulus on tangent plane
\newcommand{\B}{\mathcal{B}}%ball on tangent plane
\newcommand{\Manoa}{M\=anoa}
\newcommand{\Hawaii}{Hawai\kern.05em`\kern.05em\relax i }
\newcommand{\vt}{\tilde{v}}

\newcommand{\tp}{\widetilde{p}}
\newcommand{\K}{\mathrm{K}}
\newcommand{\Om}{\widetilde{\Omega}}
\newcommand{\spn}{\mathrm {span}}
\newcommand{\ch}{\raisebox{2pt}{$\chi$}}

\def\Lam{\varLambda}
\def\caln{{\mathcal N}}
\def\cN{{\mathcal N}}

\def\bfa{{\bf a}}
\def\bfc{{\bf c}}
\def\bfy{{\bf y}}
\def\bfA{{\bf A}}
\def\V{{\widetilde{V}}}
\def\G{{\mathrm{G}}}
\def\wbfA{{\widetilde{\bfA}}}
\def\wA{{\widetilde{A}}}

\numberwithin{equation}{section}
\title{An inverse theorem for compact Lipschitz regions in $\RR^d$ using localized kernel bases
}
\author{T.Hangelbroek\thanks{ 
Department of Mathematics, University of \Hawaii   -- \Manoa,
2565 McCarthy Mall,
Honolulu, HI, USA. Research supported
by grant DMS-1413726  from the National Science Foundation.},
F. J. Narcowich\thanks{ Department of Mathematics, Texas A\&M
   University, College Station, TX 77843, USA. Research
    supported by grant DMS-1514789 from the National   Science Foundation.}, 
 C. Rieger\thanks{Institut f{\" u}r Numerische Simulation,
Universit{\"a}t Bonn,
 Wegelerstr. 6,
53115 Bonn, Germany. Research supported
by Collaborative Research Centre (SFB) 1060 of the Deutsche Forschungsgemeinschaft.},
J. D. Ward\thanks{ Department of Mathematics, Texas A\&M University,
    College Station, TX 77843, USA. Research supported by
    grant DMS-1514789 from the National Science
  Foundation.}  }

\begin{document}
\maketitle
\begin{abstract}
While inverse estimates in the context of radial basis function approximation on boundary-free domains 
have been known for at least ten years, 
such theorems for the more important and difficult setting of bounded domains have been notably absent. 
This article  develops inverse estimates for finite dimensional spaces arising in radial basis function approximation 
and meshless methods. 
The inverse estimates we consider control Sobolev norms of linear combinations of a localized basis  
by the $L_p$ norm over a bounded domain. 
The localized basis is generated by forming  local Lagrange functions 
for certain types of RBFs (namely Mat{\'e}rn and  surface spline RBFs).
In this way it extends the boundary-free
construction recently presented in \cite{FHNWW}. 
\end{abstract}
\section{Introduction} 
This article presents a construction for localized bases generated by radial basis functions (RBFs)
in the presence of a boundary
and develops analytic properties of this basis, most notably inverse inequalities.
Such inequalities are an essential tool
in the numerical solution of PDEs by 
finite element and related methods (see \cite{Dahmen, Graham, Georgoulis}) 
notably in proving $\inf$-$\sup$ 
(Babu{\v s}ka-Brezzi) conditions,  which play a central role for mixed element and saddle point problems \cite{Bren, guermond, guzman, melenk}.
They are also prevalent in approximation theory (where they are called ``Bernstein inequalities");  specifically
they are used to obtain characterization of approximation spaces as interpolation spaces by way of $K$-functionals \cite{nonlinear}.

The type of localized basis investigated in this article has been introduced very recently for the boundary-free setting 
(e.g., on a manifold without boundary)
and has already been employed 
to  deliver strong results in function approximation and scattered data fitting \cite{FHNWW}, 
numerical quadrature \cite{FHNWW2} and solution of PDEs \cite{NRW} and integral equations \cite{LR}. 
Indeed, in \cite{LR}, Lehoucq and Rowe have applied the localized basis investigated in this article  to 
obtain a Galerkin solution to a constrained integral equation, and  they have used the $L_p$ stability of the basis 
(presented in this paper in Section 4)
to obtain norm bounds on the stiffness matrix associated with this problem.

The inverse estimates we consider treat finite dimensional spaces of functions, bounding strong (Sobolev) norms 
by weak (Lebesgue) norms:
\begin{equation}\label{template}
\|s\|_{W_p^{\sigma}(\Omega)} \le C h^{-\sigma} \|s\|_{L_p(\Omega)}\qquad\text{(or }\|s\|_{C^{\sigma}(\Omega)} \le C h^{-\sigma} \|s\|_{L_{\infty}(\Omega)}\ \text{for }p=\infty\text{)},
\end{equation}
where $\Omega$ is a bounded subset in $\RR^d$, subject to mild conditions on $\partial \Omega$
and 
$h$ is the fill distance (also known as mesh ratio) of the finite set of points used to generate our finite dimensional
space (see Section \ref{point sets} for a precise description).
In one sense, these estimates can be viewed as providing an  operator norm bound (from $L_p \to L_p$)
of differential operators restricted to  this finite dimensional space.
In another sense, they give precise equivalences between different norms in terms of a 
simple measure of the complexity (given by the parameter 
%$h$
$N$ above) of the 
finite dimensional space.
Direct consequences of these inverse estimates include trace estimates and Bernstein-Nikolskii inequalities.

This topic has been considered in the boundary-free setting by a number of authors, 
we list \cite{NWW_Bernstein}, \cite{rieger2008sampling}, \cite{MNPW}, \cite{Ward_J}, \cite{griebel2013multiscale} (although there are certainly others).
The  inequalities we consider here are similar, but 
depend only on the norm of a basic function over a bounded region\footnote{
A previous result in the setting of a bounded region was presented in \cite{SW_Inv}, but these
estimates significantly undershoot the precise exponent $-\sigma$ in (\ref{template}).}. 
Without a doubt this type of estimate is significantly more challenging when
a boundary is present and has, to the best of our knowledge, remained elusive.
Indeed, such inverse inequalities seem to have been absent for meshless methods in general
(not only radial basis function approximation, cf. the discussion in \cite[Section 7]{melenk}). 

%%%%%%%%%%%%%
%
%Overview of construction
%
%%%%%%%%%%%%%%
In this article we consider 
two prominent  families of radial basis functions: 
the  Mat{\'e}rn (or Whittle-Mat{\'e}rn) and surface spline kernels. 
Generalizations to other kernels and other settings (namely, compact Riemannian manifolds)
are fairly straightforward, but complicated. They have been considered in the manuscript \cite{HNRW1}.

The conventional finite dimensional space associated with a positive definite RBF $\phi$ and a finite set $X \subset \RR^d$ has the form 
$S(X)= \mathrm{span}_{\eta\in X} \phi(\cdot-\eta)$; 
for a conditionally positive definite RBF, $S(X)$ involves polynomials; see Section~\ref{SS:ss_kernels}. 
A common set-up for a host of numerical problems invites the user to employ the  basis of sampled kernels $\phi(\cdot-\eta)$, $\eta\in X$ as one would use 
polynomials, splines, finite elements, etc.: that is to say as  test functions for Galerkin or collocation methods, or as basis functions to solve interpolation, quadrature or
other basic problems. 

For a basic interpolation problem, using $S(X)$ to interpolate data sampled at the point set $X$, the ensuing interpolation matrix will be positive definite, 
thanks to the kernel's positive definiteness, but if $X$ is  sampled densely, the interpolation matrix will become dense\footnote{One may attempt to circumvent this problem by 
dilating the kernel; this is often done, but will generally result in degraded rates of approximation.}.

Instead of using the basis of kernels, one may attempt to use another basis for $S(X)$; one for which basic matrices (Gram, collocation, stiffness, interpolation) exhibit off-diagonal decay. 
Univariate splines provide a prime example of this phenomenon: for a fixed $k$, the shifted truncated powers $(x-t_j)\mapsto x_+^k$ provide, in conjunction with polynomials of degree $k$ or less,
a basis for the spline space with breakpoints at $t_j$, but this basis is known to be poorly localized. However, the $B$-spline basis is well-localized, with elements having support which is not only
compactly supported, but {\em stationary} in the sense that it shrinks with the spacing of the breakpoints.

We are concerned with an analogous
{\em localization problem}  for radial basis functions:
\begin{quotation}
Is there a basis for $S(X)$ where the various elements exhibit a fast rate of {\em stationary} decay?
\end{quotation}
If an alternative basis is available for which the interpolation matrices are sparse, we say the basis is well-localized.
For the Mat{\'e}rn and surface-spline kernels,
the Lagrange function $\chi_\eta$  is well localized %, 
  in a neighborhood of  $\eta$ where the points from $X$ are distributed quasi-uniformly.
If this is not the case, for instance if $\eta$ occurs near to the boundary  of the convex hull of $X$, localization is lost. 

This issue
 can be circumvented by using only the Lagrange basis elements $\chi_{\xi}$ that have centers $\xi$ in a sufficiently large subset  
$\Xi \subset X$, where  $\Xi$ is chosen so that the Lagrange functions $\chi_\xi$, $\xi \in \Xi$, are localized. 
Using these elements we may define $V_\Xi:=\mathrm{span}_{\xi \in \Xi} \chi_{\xi}$, which is of course a subspace of $S(X)$. 
To avoid a possible point of confusion, we  emphasize that $V_\Xi \ne S(\Xi)$. 
The former space requires all basis functions centered in $X$ for its construction, the latter only those in $\Xi$.

After this initial streamlining,
it is important to note that even though $\chi_\xi$, $\xi \in \Xi$, is spatially localized, 
its construction still requires \emph{all} of the points in $X$. 
Thus finding the $\chi_\xi$'s is computationally expensive. 
In \cite{FHNWW}, \emph{local} Lagrange functions $\{b_\xi\}_{\xi \in \Xi}$ were introduced. 
Constructing them is done by first choosing points $\Upsilon(\xi) \subset X$ in a small neighborhood  of $\xi\in \Xi$,  
and then finding the Lagrange function $b_\xi \in S(\Upsilon(\xi)) \subset S(X)$. 
Since $\Upsilon(\xi)$ will contain many fewer points then $X$, it will be much less expensive to find $b_\xi$. 
Finally, we define $\V_{\Xi} =  \mathrm{span}_{\xi \in \Xi} b_{\xi}$, which is a subspace of $S(X)$. 
We remark that $\chi_\xi \ne b_\xi$ and $V_{\Xi} \ne \V_{\Xi}$. However, they are close -- a fact that will prove important in the sequel. 

We now turn to the connection between the set $\Omega$ and the spaces described above. 
At the start, we are given a quasi-uniform set $\Xi \in \Omega$. 
The enlarged set $X$ is \emph{not} given. 
Rather, an extension is constructed from $\Xi$, using a method -- described in Section~\ref{SS:Extending_points} -- 
that preserves the 
%quasi-uniformity 
key geometric properties 
of $\Xi$. 
The extension, which will be denoted by $\widetilde \Xi$ later (instead of $X$),  
is contained in a bounded region $\Om$ that contains $\Omega$ and is roughly speaking about twice the size of $\Omega$. 
It is for this setup that we get estimates of the form (\ref{template}) for $s\in V_{\Xi}$ or $\V_{\Xi}$. (See Theorem \ref{main}.)

\subsection{Overview and Outline} 
We begin by giving basic explanation and background on RBFs used in this article. This is done in Section 2.

In Section 3, we introduce the Lagrange basis (the functions generating the space $V_{\Xi}$)
and  provide estimates that control the Sobolev norm (i.e., $W_p^{\sigma}(\Omega)$) 
of a function in $V_{\Xi}$ by the $\ell_p$ norm on the Lagrange coefficients and in addition by the $L_p$ norm of $s$. That is, for $s= \sum_{\xi\in \Xi} a_{\xi} \chi_{\xi}$ we show 
$$
\|s\|_{W_p^{\sigma}(\Omega)}\le C (\#\Xi)^{1/p-\sigma/d}\|(a_{\xi})_{\xi\in\Xi}\|_{\ell_p(\Xi)} 
\quad \text{and}\quad 
\|s\|_{W_p^{\sigma}(\Omega)}\le C (\#\Xi)^{-\sigma/d}\|s\|_{L_p(\Omega)}.
$$
Such a result has not appeared previously.

Section 4 introduces the other stable basis considered in this paper: the local Lagrange basis, 
which generates the space $\V_{\Xi}$. 
We give sufficient conditions to prove existence and stability of such a basis. We give estimates that control the Sobolev norm (i.e., $W_p^{\sigma}(\Omega)$) 
of a function in $\V_{\Xi}$ by the $\ell_p$ norm on the local Lagrange coefficients and by the $L_p$ norm of the function.
This result  is presented in Theorem~ \ref{main_local_bernstein}. Next we compare the sequence norm with the $L_p$ norm of 
an expansion
 $s = \sum_{\xi\in\Xi} a_{\xi} \chi_{\xi} \in V_{\Xi}$  or $s = \sum_{\xi\in\Xi} a_{\xi} \chi_{\xi} \in \V_{\Xi}$ 
 over the domain $\Omega$. We thus obtain
$$\|(a_{\xi})_{\xi\in\Xi}\|_{\ell_p(\Xi)}\sim C(\# \Xi)^{-1/p} \|s\|_{L_p(\Omega)}.$$

In the final section we give our main inverse estimates. For $s\in \V_{\Xi}$ we have
$$
\|s\|_{W_p^{\sigma}(\Omega)} \le C (\#\Xi)^{-\sigma/d} \|s\|_{L_p(\Omega)},
$$
and we use this to demonstrate trace estimates for that space.

\section{Background: RBF approximation on bounded domains}
We begin by describing the basic elements used in this article, 
starting with geometric properties of point sets, 
a discussion of the the underlying  domain,
smoothness spaces on the domain,
and finishing with some background about the radial basis functions which we use.

\subsection{Point sets}\label{point sets}
 Given a set $D\subset \reals^d$ and  a discrete, possibly infinite, set 
 $X \subset D$, we define its  \emph{fill distance} 
 (or \emph{mesh norm}) $h$, the \emph{separation radius} $q$ and 
 the \emph{mesh ratio} $\rho$ to be:
\begin{equation} \label{minimal-separation}
 h(X,D):=\sup_{x\in D} \d(x,X), \quad    q(X):=\frac12 \inf_{\xi \in X} 
\d(\xi, X\setminus \{\xi\}) , \quad  \rho(X,D):=\frac{h(X,D)}{q(X)},
\end{equation}  
where in defining $\rho(X,D)$ we assume that $q(X)>0$. 

%\begin{remark} 
When there is no chance of confusion, we drop dependence in these parameters on $X$ and $D$ (referring simply to $h,q$ and $\rho$).
%\end{remark} 

\begin{remark}
A finite fill distance $h$
guarantees that the set $D$ is covered by the  family of balls $B(\xi,h): = \{x\in D\mid \dist(x,\xi)< h\}$, $\xi\in X$. 
A positive separation radius $q$ guarantees that $B(\xi,q)\cap B(\zeta,q) = \{\}$ for distinct $\zeta,\xi \in \Xi$.
 The mesh ratio, which automatically satisfies  $\rho\ge 1$, measures  the uniformity of the 
distribution of $X$ in $D$. 
The larger $\rho(X,D)$ is, the less uniform the distribution is. If $\rho$ is finite, %``small'', 
then we say that the point set $X$ is quasi-uniformly distributed (in $D$), or simply that $X$ is quasi-uniform. 

Note that, for a compact subset $D$ and a nonempty, finite subset $X\subset D$, the fill distance and separation radius are both positive and finite $0<q<h<\infty$. Consequently, $\rho$ is finite, too.

Many of the results in this article depend in some way on the geometry of the point set $X$ -- 
often this emerges in an estimate, where a constant depends on $\rho$. In most cases, (as one may expect)
the strength of the estimate degrades as $\rho$ increases. Throughout the paper, we have attempted make this control explicit, by factoring, whenever possible, the constant into a part which 
is  totally independent of the point set, and another, which is a function of $\rho$.
\end{remark}

It is often useful to estimate certain sums over $X$.  Assume that $q(X)>0$. If $f:[0,\infty) \to [0,\infty)$
is  a positive, decreasing, continuous function, then 
 \begin{equation}\label{decreasing}
\sum_{\zeta\in X} f(\dist(\zeta,\xi))  \le f(0) + C\sum_{n=1}^\infty n^{d-1} f(nq)
\end{equation}
where $C$ depends only on the spatial dimension $d$. 
This is easily established by introducing annuli centered at $\xi$, with inside radius $nq$ and outside radius $(n+1)q$, $n\ge 1$. 
The number of points contained in each annulus is proportional to $n^{d-1}$, and the contribution to the sum from each $n$,  $n\ge 1$, is less than $n^{d-1}f(nq)$. Hence, \eqref{decreasing} holds. 
 
 \subsection{The domain $\Omega$ }\label{bounded_domain}
We now consider a bounded region $\Omega \subset \reals^d$ containing a finite point set $\Xi$ with
$h=h(\Xi,\Omega)$ and $q = q(\Xi)$ as defined above.
This presents two challenges. 

The first concerns  $\Xi$ -- although we may expect it to be finely sampled (often referred to as {\em sufficiently dense}, meaning that $h(\Xi,\Omega)$ is small) in $\Omega$, it will not be so 
in a neighborhood of $\Omega$. 
To construct the localized bases to be used in the sequel, 
we need a larger set $X\subset \reals^d$ 
so that $X\cap \Omega = \Xi$. 
In other words, we require some extra points to  lie outside of $\Omega$ 
(in fact, when working with local Lagrange functions $b_{\xi}$, 
it suffices to consider only a very small extension 
$\Upsilon \subset \{x\in \reals^d \mid \dist(x,\Omega)<K h |\log h|\}$).
This assumption is in place to guarantee decay of the basis functions -- in other words, it is only a tool for guaranteeing 
the decay of $\chi_{\xi}$ or $b_{\xi}$, and is not otherwise important for the stability estimate. 
It would be quite reasonable to be `given' initially only the set $\Xi\subset \Omega$  and to use this to 
construct $X$ .
In Lemma \ref{Extension} below we
demonstrate how to extend a given set of centers $\Xi\subset \Omega$ in a controlled way to obtain a 
satisfactory set $X$.

The second challenge concerns the domain $\Omega$.
For estimates relating $\|\bfa \|_{\ell_p}$ and the $L_p$ norm of expansions $\|\sum_{\xi} a_{\xi} b_\xi\|$
or $\|\sum_{\xi} a_{\xi} \chi_\xi\|$ the boundary becomes  more important.
The extra assumption we make on $\Omega$, in force throughout the article, is that
$\M$ satisfies an interior cone condition (see Appendix \ref{Appendix_A} for a discussion).

\subsection{Extending points}\label{SS:Extending_points}
Given $\M$  and $\Xis \subset\M$, 
we wish to find an extension $\wX \supset \Xis$ 
dense in $\reals^d$
so that 
the separation radius does not decrease and the fill distance is controlled (and, consequently, $\rho$ does not increase). A simple constructive example is the following.
\begin{lemma}\label{Extension}
Suppose $\Xi\subset \M$ has fill distance $h(\Xis,\M) = h$ and separation radius $q(\Xis) = q$.
Then there is a discrete set $\wX$ so that   $\wX\cap \Omega = \Xi$,
$q({\wX}) = q,$
and $h(\wX,\reals^d) =  h(\sqrt{d}/2+2)$.
\end{lemma}
\begin{proof}
We proceed as follows: let $\wX = \Xi\cup\{\zeta \in h\ints^d \mid \dist(\zeta, \Omega) \ge h\}$. 
We note that $h(h\ints^d, \reals^d) = \frac{\sqrt{d}}{2}h$
and $q(h\ints^d) = h$. 
It follows immediately that $q(\wX) = q$. If $x \in \reals^d$ is within $(\frac{\sqrt{d}}{2}+1)h$ of $\Omega$, 
then $\dist(x, \Xi) \le (\frac{\sqrt{d}}{2}+2)h$.
On the other hand, if $x \in \reals^d$ satisfies
$\dist(x,\Omega)>(\frac{\sqrt{d}}{2}+1)h$
then there is $\zeta \in h\ints^d$ with $\dist(x,\zeta) < \frac{\sqrt{d}}{2}h$
 so that $\dist(\zeta,\Omega) >h$ (and $\zeta$ is therefore in $\wX$).
\end{proof}

\begin{figure}[!ht]

\centering
\includegraphics{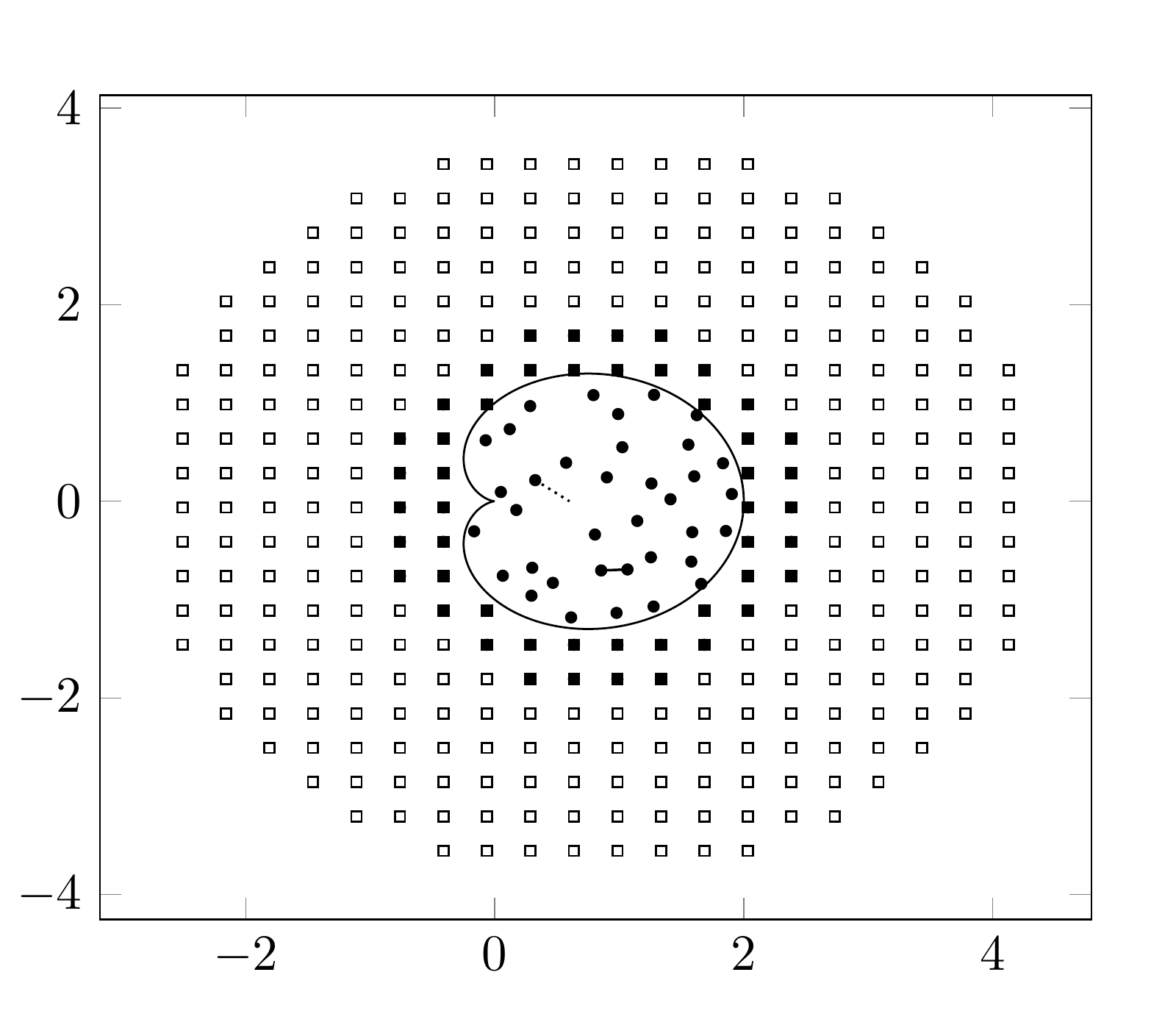}
 \caption{This image shows the domain $\Omega$ (the region inside the cardioid) with a set of points $\Xi\subset \Omega$
 indicated with $\bullet$. The dotted line segment indicates $h$ - the greatest distance between a point of  $\Omega$ and one of $\Xi$.
 The solid line segment indicates $2q$, the nearest neighbor distance in $\Xi$.
 The elements in the  extended point set 
 $\X\setminus \Xi$ are denoted with a square -- these are the centers used to construct $\chi_{\xi}$ (discussed in Section 3). 
 The points $\blacksquare$ denote the points of $\Upsilon$, which are used to construct $b_{\xi}$ (this is done in Section 4).}
 \label{figure2}
 
\end{figure}

\begin{remark}\label{True_Extension}
We note that other extensions exist which do not increase $h$. For example, \cite[Lemma 5.1]{HNRW1} extends points
so that $h(\wX,\reals^d)=h$ and $q(\wX) = \min(q, h/2)$. 
As an expository convenience, we use an extension $\wX$ of $\Xi$ 
to $\reals^d$ which does not increase $h$. In practice, an extension could be used which might not precisely preserve the geometry of the point set (such as the elementary one in Lemma \ref{Extension}). This will not change the results in this paper, other than by modifying slightly the constants.
We leave it to the reader to make the (very simple) modifications necessary to treat other extensions (which would increase $h$ and $\rho$).
\end{remark}

We construct the extended point set  in an extended neighborhood 
\begin{equation}\label{finite_extension} 
\X := \wX\cap\Om
\quad \text{where}\quad
\Om := \{x\in \reals^d\mid \dist(x,\Omega)\le \diam (\Omega)\}
\end{equation}
and where $\wX$ is constructed according to the method of 
Remark \ref{True_Extension}.

\subsection{Smoothness spaces on $\Omega$}\label{SS:Sobolev}
In order to present a suitably robust family of inverse estimates, we employ a scale of spaces depending on a positive, occasionally fractional, smoothness parameter; as in \cite{Bren}, 
for integer values of this parameter, we use the conventional Sobolev spaces, while for fractional values we use fractional spaces, which involve a H{\" o}lder-like seminorm.

For a domain $\Omega\subset \reals^d$, the integer order Sobolev space is defined in the conventional way.
For $1\le p<\infty$ and $m\in \nats$, we have the semi-norm and norm 
$$|u|_{W_p^m(\Omega)}^p : = \sum_{|\alpha| = m} \begin{pmatrix} m\\ \alpha \end{pmatrix} \int_{\Omega} |D^{\alpha}u(x)|^p\,\mathrm{d}x, \qquad
\|u\|_{W_p^m(\Omega)}^p : = \sum_{k=0}^m \begin{pmatrix} m\\ k\end{pmatrix} |u|_{W_p^k (\Omega)}^p.
$$
Note that for the first expression (the Sobolev semi-norm), we use the binomial coefficient  with multi-integers $\begin{pmatrix} m\\ \alpha \end{pmatrix}  =\frac{m!}{\alpha_1!\dots \alpha_d!}$ 
while for the second  we use a standard binomial coefficient 
$\begin{pmatrix} m\\ k \end{pmatrix}  =\frac{m!}{k!(m-k)!}$. Although other weights would give equivalent norms, resulting in the same Sobolev spaces equipped with the same topology, these choices of coefficients will be necessary to obtain the specific reproducing kernels we desire
(see Section \ref{SS:M_kernels} and \ref{SS:ss_kernels}).

For fractional orders $\sigma = m+\delta\notin \nats$ with $0<\delta<1$
we add the Slobodeckij semi-norm 
$$
|u|_{W_p^{\sigma}(\Omega)}^p
:=\sum_{|\alpha|=m} \int_{\Omega}\int_{\Omega} \frac{|D^{\alpha}u(x) - D^{\alpha}u(y)|^p}{|x-y|^{d+p\delta}}\,\dif x \dif y,
\qquad
\|u\|_{W_p^\sigma(\Omega)}^p 
: = \|u\|_{W_p^m(\Omega)}^p+
|u|_{W_p^{\sigma}(\Omega)}^p.
$$
We note that when $\sigma=m+\delta$ is 
fractional\footnote{When $\sigma = m$ is an integer, we have $W_2^m(\Omega) = B_{2,2}^{m}(\Omega)$, although $W_p^m(\Omega) \subsetneq B_{p,p}^{m}(\Omega)$ for $p>2$ and $W_p^m(\Omega) \supsetneq B_{p,p}^{m}(\Omega)$ for $p<2$.}, $W_p^{\sigma}(\Omega)$ is the Besov space 
$B_{p,p}^{\sigma}(\Omega)$ 
(this is \cite[Theorem 6.7]{DevSh}). In particular, $W_p^{\sigma}(\Omega) = B_{p,p}^{\sigma}(\Omega)= [W_p^m(\Omega),W_p^{m+1}(\Omega)]_{\delta,p}$ serves as the  $[\delta,p]$ (real) interpolation space between
$W_p^m(\Omega) $ and $W_p^{m+1}(\Omega)$ (see \cite[1.6.2]{Trieb2} for a definition and basic results).

Of particular importance is the fact that, for $2\le p<\infty$ and $m\in \nats$, we have the continuous embedding 
$W_2^m(\Omega)\subset W_p^s(\Omega)$ for all $s\le m - (d/2 - d/p)$.

Throughout the paper, we make the (not unusual) modification 
$W_{\infty}^m(\Omega) = C^m(\overline{\Omega})$ when $p=\infty$ and $m\in \nats$. 
For fractional order spaces when $p=\infty$ (discussed in Section 5),
we use the H{\" o}lder space $C^{s}(\overline{\Omega})$, for which 
$ \max_{|\alpha| = \lfloor s\rfloor} \sup_{x,y\in\overline{\Omega}} \frac{|D^{\alpha}u(x) - D^{\alpha}u(y)|^p}{|x-y|^{\delta}}$ 
is finite for $\delta = s -\lfloor s\rfloor$. In this case,
$W_2^m(\Omega)\subset W_p^s(\Omega)$ for all $s< m - d/2 $.

\subsubsection{Scaling and fractional Sobolev spaces}
For an open set $O\subset \reals^d$, let us introduce the notation $O_R := \{x\mid x/R \in O\}$. The following lemma shows how the fractional Sobolev seminorm scales with $R$. 
\begin{lemma}\label{scaling}
Suppose $1\le p\le\infty$, $s\in [0,\infty)$ and $u\in W_p^s(O_R)$. Let $U:O\to \comps: x\mapsto u(Rx)$. Then
$$ |u|_{W_p^s(O_R)}=  C R^{d/p - s} |U|_{W_p^s(O)}
$$
\end{lemma}
\begin{proof}
We consider the case $1\le p<\infty$ and $s = k+\delta$, $0<\delta<1$, 
since the cases where $s$ is an integer and $p=\infty$  follow similarly, but are much easier.
 For $RX = x$, the chain rule gives us $D^{\alpha} u(x) = R^{-|\alpha|} U(X)$ and 
\begin{eqnarray*}
\sum_{|\alpha|=k} \int_{O_R} \int_{O_R}\frac{|D^{\alpha}u(x)-D^{\alpha}u(y)|^p}{|x-y|^{d+p\delta}} \dif y\dif x
&=&
R^{d - p\delta - pk} 
\sum_{|\alpha|=k} \int_{O} \int_{O}\frac{|D^{\alpha}U(X)-D^{\alpha}U(Y)|^p}{|X-Y|^{d+p\delta}} \dif Y\dif X\\
&=&
R^{d-ps} |U|_{W_p^s(O)}^p . 
\end{eqnarray*}
\end{proof}

\subsubsection{Sub-additivity and fractional Sobolev spaces}

Carstensen and Faermann \cite{carstensen2001mathematical} have pointed out that the $p$th power 
$|u|_{W_p^{\sigma}(\Omega)}^p$
of the fractional Sobolev seminorm fails to be sub-additive. This is in contrast to the ($p$th power) integral order seminorms,
which are obtained from integrals of non-negative functions, and are easily seen to be sub-additive.

The following lemma is a modification of a result of Faermann (\cite[Lemma 3.1]{Faermann2002}) 
which we use as a tool to treat the issue of non-subadditivity. This will be used in the sequel.

\begin{lemma}\label{sa}
Suppose $\mathcal{V}  = \{\vt_j\mid j\in \cN\}$ is a countable family of subsets $\vt_j\subset  \Omega$
covering $\Omega$ with finite overlap: i.e., $\Omega\subset \bigcup_{j\in \cN} \vt_j$ 
and  there is $M>0$ so that $\max_{x\in \Omega} \sum_{j\in \cN} \chi_{\vt_j}(x) \le M$.
Suppose further that there exist sets $v_j \supset \vt_j$ so that the 
complements $w_j :=  \Omega\setminus v_j$ each 
 are a fixed positive distance from the corresponding sets $\vt_j$: 
 i.e., there  is $H>0$ so that for every $j\in \cN$, $\inf_{x\in \vt_j, y\in w_j} |x - y| \ge H$.
 
Let $1\le p<\infty$ and $s\in (0,\infty)\setminus \mathbb{N}$ with 
$k = \lfloor s\rfloor$ and $\delta = s-\lfloor s\rfloor$. Then for any $u\in W_p^s(\Omega)$ we have
\begin{equation}\label{partial_subadditivity}
\|u\|_{W_p^s(\Omega)}^p 
\le 
\left(
\sum_{j\in \cN} |u|_{W_p^s(v_j)}^p
\right)
+ C M H^{-p\delta} \|u \|_{W_p^k(\Omega)}^p
\end{equation}
\end{lemma}
\begin{proof}
By sub-additivity of the outer integral,  we  have that 
\begin{eqnarray*}
|u|_{W_p^s(\Omega)}^p  
&\le& 
\sum_{|\alpha|=k}\sum_{j\in \cN} \int_{\vt_j} \int_{\Omega}\frac{|D^{\alpha}u(x)-D^{\alpha}u(y)|^p}{|x-y|^{d+p\delta}} \dif y\dif x\\
&\le&
\sum_{|\alpha|=k}
 \sum_{j\in \cN}
 \left(
 \int_{\vt_j} \int_{v_j}
 \frac{|D^{\alpha}u(x)-D^{\alpha}u(y)|^p}{|x-y|^{d+p\delta}} \dif y\dif x\right. 
 + 
\left. \int_{\vt_j} \int_{w_j}\frac{|D^{\alpha}u(x)-D^{\alpha}u(y)|^p}{|x-y|^{d+p\delta}} \dif y\dif x \right).
\end{eqnarray*}
The first terms are controlled by $\sum_{j=1}^{\infty} |u|_{W_p^s(v_j)}^p$, since $\vt_j\subset v_j$, 
and so this gives the first part of the right hand side of (\ref{partial_subadditivity}).

Consider the sum of the second terms. 
Applying the quasi-triangle inequality $(a+b)^p \le C_{p} (a^p +b^p)$ to the numerator $|D^{\alpha}u(x)-D^{\alpha}u(y)|^p$,
we obtain
\begin{eqnarray*}
\int_{\vt_j} \int_{w_j}\frac{|D^{\alpha}u(x)-D^{\alpha}u(y)|^p}{|x-y|^{d+p\delta}} \dif y\dif x 
&\le& C_p\left(\int_{\vt_j} \int_{w_j}\frac{|D^{\alpha}u(x)|^p}{|x-y|^{d+p\delta}} \dif y\dif x  
+ 
 \int_{\vt_j}  \int_{w_j}\frac{ |D^{\alpha}u(y)|^p}{|x-y|^{d+p\delta}} \dif x\dif y\right) \\
   &=:& J_{j,1} + J_{j,2}.
 \end{eqnarray*}
We have that $\bigcup_{j\in \cN} (\vt_j\times w_j) \subset \{(x,y)\in\Omega^2\mid
 |x-y|>H\}$.
By symmetry, we have also that 
$\bigcup_{j\in \cN} (w_j\times\vt_j) \subset \{(x,y)\in\Omega^2\mid 
|x-y|>H\}$.
Using the finite overlap, 
 we have that
  $$
 \begin{rcases*}
  \sum_{j\in \cN} \raisebox{1pt}{$\chi$}_{[\vt_j\times w_j]}\\
   \sum_{j\in \cN} \raisebox{1pt}{$\chi$}_{[ w_j\times \vt_j]}
   \end{rcases*}
  \le M \ch_{\{(x,y)\in\Omega^2\mid 
  |x-y|>H\}}.
  $$
  Consequently, for any non-negative, integrable $g: \Omega\times \Omega \to \RR$, we have
$$
\begin{rcases*}
\sum_{j\in \cN} \int_{\vt_j}\int_{w_j}g(x,y)\dif x \dif y\\
 \sum_{j\in \cN}\int_{w_j}\int_{\vt_j} g(x,y)\dif x \dif y
 \end{rcases*}
\le 
M \int_{ \{(x,y)\in\Omega^2\mid 
|x-y|>H\}} g(x,y)
$$ 
as well.
Setting $g(x,y) = \frac{|D^{\alpha}u(x)|^p}{|x-y|^{d+p\delta}}$ and applying Fubini - Tonelli 
allows us to control 
 $\sum_{j\in \cN} (J_{j,1} +J_{j,2})$
by
$$
\sum_{j\in \cN} (J_{j,1} +J_{j,2})
\le 2M  
\int_{ \{(x,y)\in\Omega^2\mid 
|x-y|>H\}}
\frac{ |D^{\alpha} u(x)|^p}{|x-y|^{d+p\delta}}
\le  CM  H^{-p\delta}   \int_{\Omega} |D^{\alpha}u(x)|^p  \dif x .
$$
(For the last estimate, we have used the fact that $\int_{|x-y|>H} \frac{1}{|x-y|^{d+p\delta} } \dif y \le C_{p\delta} H^{-pd}$
for all $x$.)
\end{proof}

\subsection{Radial basis functions}\label{S:Kernels}
There are two families of radial basis functions considered in this article: the Mat{\'e}rn functions and the surface splines.
Both families (under the right conditions) admit exponentially decaying basis functions -- 
this is mentioned in Section \ref{SS:Lagrange}.
They also admit rapidly constructed localized basis functions (having polynomial decay) --
this is demonstrated in Sections \ref{SSS:Local_Matern} and \ref{SSS:Local_surface}.
The results we present in Sections \ref{main_results} hold for these families.

Two  features common to both families are:
\begin{enumerate}
\item For any  finite set of points $\Xi\subset \reals^d$ the interpolation problem is well posed. 
This means that for any data $(\xi, y_\xi)_{\xi\in\Xi}$, there exists a unique interpolant $s$ generated by the RBF.
\item The RBF is a reproducing kernel for a (semi-)Hilbert space, called the {\em native space}, 
and the unique interpolant  to $(\xi, y_\xi)_{\xi\in\Xi}$ is 
the {\em best} interpolant in this space: it has the least (semi-)norm among all interpolants to the data.
\end{enumerate}
We include both families (which are in some ways quite similar) because both are often in use, practically. The first is prized for the RBF's rapid decay and strict positive definiteness; the second is included for its dilation invariance and its historical significance. Of course, there are many other prominent families of RBFs, each with its own distinguishing features (some are infinitely smooth, some are compactly supported, etc.). 
Rather than give a broad overview, we introduce the specific families employed in this paper and direct the interested reader to  \cite{Wend} for a comprehensive introduction to RBF theory.
At this point it is unclear if the algorithm for constructing localized bases works for other families; the arguments we employ rely heavily on the RBF's role as the fundamental solution to an elliptic partial differential operator.

\subsubsection{Mat{\'e}rn kernels}  \label{SS:M_kernels}
The Mat{\' e}rn function of order $m>d/2$ is defined as
\begin{equation}
\label{Matern}
\kappa_m:\reals^d  \to \reals:x\mapsto C
K_{m-d/2}(|x|)\,|x|^{m-d/2}.
\end{equation}
Here $C$ is  a constant depending on $m$ and $d$, and $K_{\nu}$ is a Bessel function of the second kind. 

The Mat{\' e}rn function is positive definite, which means that for every finite set $X \subset \reals^d$, the {\em collocation} matrix 
$$\K_{X} :=(\kappa_m(\xi - \zeta))_{\xi,\zeta\in X}$$ 
is strictly positive definite. 

The guaranteed  invertibility of $\K_{X}$ is of use in solving interpolation problems -- given $\bfy\in \reals^{X}$, one finds $\bfa\in\reals^{X}$ 
so that $\K_{X} \bfa = \bfy$. It follows that $\sum_{\xi\in X} a_{\xi} \kappa_m(\cdot - \xi)$ is the unique interpolant to $({\xi},y_{\xi})_{\xi\in X}$ in 
$S(X) := \mathrm{span}_{\xi\in X}\kappa_m(\cdot - \xi)$. 

It is the reproducing kernel for the Hilbert space $\caln(\kappa_m) = W_2^m(\reals^d)$ equipped with the (standard) inner product
$$\langle f,g \rangle_{W_2^m(\reals^d)} 
= 
\int_{\reals^d} \sum_{j=0}^m\begin{pmatrix} m\\ j\end{pmatrix} \langle D^j f(x), D^j g(x) \rangle  \dif x 
= 
\sum_{|\beta|\le m} 
 \begin{pmatrix} m\\ |\beta|\end{pmatrix}  \left( \begin{matrix} |\beta|\\\beta\end{matrix}\right) \int_{\reals^d}  D^{\beta}f(x) D^{\beta} g(x)  \dif x 
$$
where $D^j f$ is the tensor (i.e., the  $j$-dimensional array) of partial derivatives of order $j$. 
Being the reproducing kernel means simply that
$f(x) = \langle f, \kappa_m(x-\cdot)\rangle_{W_2^m(\reals^d)}$ for all $x\in \reals^d$ and all $f\in {W_2^m(\reals^d)}$.
It can be shown that  
among all functions interpolating the data $({\xi},y_{\xi})_{\xi\in X}$,
the interpolant $\sum_{\xi\in X} a_{\xi} \kappa_m(\cdot - \xi)$
 (i.e., where $\bfa$ is the solution of $\K_X \bfa = \bfy$) 
has the smallest
$W_2^m(\reals^d)$ norm.

\subsubsection{Surface splines}\label{SS:ss_kernels}
For $m>d/2$, the {\em surface spline}   is 
\begin{equation}\label{surface_spline}
\phi_m:\reals^d \to \reals:x\mapsto C
\begin{cases} |x|^{2m-d}&d \text{ is odd}\\
|x|^{2m-d}\log |x|&d \text{ is even}.
\end{cases}
\end{equation}

The surface spline of order $m$  is conditionally positive definite (CPD) with respect to $\Pi_{m-1}$, the space of polynomials of degree $m-1$.
This means that for every finite set 
$X \subset \reals^d$, the quadratic form 
$\reals^{X} \to \reals:\a \mapsto \langle \a , \K_{X} \a\rangle = \sum_{\xi \in X}\sum_{\zeta\in X} \phi_m(\xi-\zeta) a_{\xi} a_{\zeta}$ is positive for all nonzero 
$\a\in \reals^{X}$ satisfying
$\sum_{\xi\in X} a_{\xi} p(\xi) = 0$ for all $p\in\Pi_{m-1}$. 
(In other words, it is positive definite on a subspace of $\reals^{X}$ of finite codimension 
(namely, the annihilator of $\Pi_{m-1}\left|_{X}\right.)$. 

One may  solve interpolation problems using the finite dimensional space
$$S( X) := \left\{\sum_{\xi\in X}a_{\xi}\phi_m(\cdot- \xi)\mid \sum_{\xi} a_{\xi} p(\xi) = 0 \text{ for all } p\in\Pi_{m-1}\right\}+ \Pi_{m-1}$$ 
provided that data sites $ X$  are {\em unisolvent}: 
i.e., so that if $p\in\Pi_{m-1}$ satisfies $p(\xi) = 0$ for all $\xi\in X$ then $p=0$. 
Let $\{p_1,\dots,p_N\}$ be a basis 
for $\Pi_{m-1}$ and construct the  
$\# X\times N$ Vandermonde matrix  $\Phi = (p_j(\xi))_{\xi\in X, \, j=1,\dots,N}$.
For data $\bfy\in\reals^{ X}$ one finds $\bfa\in \reals^{ X}$ and $\bfc \in \reals^N$ so that
$$\begin{pmatrix} \K_{ X} & \Phi\\ \Phi^T & 0\end{pmatrix}\begin{pmatrix}\bfa \\ \bfc\end{pmatrix} = \begin{pmatrix} \bfy\\ 0\end{pmatrix}.$$
 It follows that $s_{ X}:= \sum_{\xi\in X} a_{\xi} \phi_m(\cdot-\xi)+ \sum_{j=1}^N c_j p_j$ is the unique interpolant to $({\xi},y_{\xi})_{\xi\in X}$ in $S( X)$.

The surface spline $\phi_m$ is the reproducing kernel 
for the semi-Hilbert space 
$$D^{-m}L_2 = \{f\in C(\reals^d) \mid \forall |\alpha| =m,\,D^{\alpha} f \in L_2(\reals^d)\}$$ (sometimes called
the  Beppo--Levi space),
which is a 
 semi-Hilbert space 
(a  vector space having a semi-definite inner product with nullspace $\Pi_{m-1}$, 
so that $D^{-m}L_2 / \Pi_{m-1}$ is a Hilbert space).
The space $D^{-m}L_2$ is endowed with the semi-definite product
$$\langle f,g\rangle_{D^{-m}L_2} = 
\langle f,g\rangle_{W_p^m(\reals^d)}
=
\int_{\reals^d} \langle D^mf(x), D^m g(x) \rangle  \dif x 
= 
\sum_{\beta = m}    \left( \begin{matrix} m \\\beta\end{matrix}\right) \int_{\reals^d}  D^{\beta}f(x) D^{\beta} g(x)  \dif x .
$$

Although $\phi_m\notin D^{-m}L_2 $ (its $m$th derivatives behave, roughly, like $\mathcal{O}(|x|^{m-d})$, which is not square integrable, since $2m>d$), 
with a little effort, one may show that the spaces $S(X)$ are contained in $D^{-m}L_2 $.
The RBF $\phi_m$ is its reproducing kernel in  the  sense that
for $ X\subset \reals^d$ and two functions $f_1, f_2\in D^{-m}L_2$ 
where  $f_2$ has the form $f_2 = \sum_{\xi \in  X} a_{\xi} \phi(\cdot-\xi) +p\in S(X)$ 
we have $\langle f_1,f_2\rangle_{D^{-m}L_2} = \sum_{\xi\in X}a_{\xi}f_1(\xi)$.
The interested reader will find a material on surface splines and conditionally positive definite RBFs in \cite[Chapter 8]{Wend}.

As in the case of Mat{\' e}rn kernels, the unique interpolant 
residing in $S(X)$
has  the smallest
$D^{-m}L_2$ semi-norm among all interpolants  to the data $({\xi},y_{\xi})_{\xi\in X}$.

\subsubsection{Labeling kernels}\label{SSS:order}
In most cases in this article, the Mat{\'e}rn and surface spline RBFs  exhibit similar
behaviors.
Because our results often depend only on a single parameter $m$ indexing the RBF, 
we use the abbreviated notation $k_m$ to stand for either
 $\kappa_m$ or $\phi_m$.
 
 In both cases, the function $k_m$ has $L_p$ smoothness less than $ 2m-d+d/p$ (i.e., for any bounded 
 set $\Omega$, $k_m \in W_p^{\sigma}(\Omega)$ for all $\sigma<2m-d+d/p$). It follows that any finite linear combination
 of shifts of $k_m$ has the same regularity.
 Denote the space of such linear combinations as
$$
S(X) := \begin{cases}
\mathrm{span}_{\xi\in X}\kappa_m(\cdot - \xi)
&k_m = \kappa_m\\
\left\{\sum_{\xi\in X}a_{\xi}\phi_m(\cdot- \xi)\mid \sum_{\xi} a_{\xi} p(\xi) 
 = 0 \text{ for all } p\in\Pi_{m-1}\right\}+ \Pi_{m-1}
 &
 k_m = \phi_m
 \end{cases}
 $$
 Then we have 
 $$S(X) \subset W_p^{\sigma}(\Omega)\quad\text{ for all }\quad \sigma<2m-d(1-1/p).$$
 
Likewise, we let $\caln(k_m)$ represent either of the two native spaces: $W_2^m(\reals^d)$ or $D^{-m}L_2(\reals^d)$.  
We note that both 
satisfy the continuous embedding 
$W_{2}^m(\reals^d)\subset \caln(k_m)\subset W_{2,loc}^m(\reals^d)$. 
In this case, the functions in the native space have a lower $L_p$ regularity, with
$$\caln(k_m) \subset\begin{cases}   W_p^{s}(\Omega) &\text{ for }\ s\le m -({d}/{2}- {d}/{p})_+, \quad 1\le p<\infty, \\
  C^{s}(\overline{\Omega})\ & \text{ for }\ s< m - d/2.\end{cases}
$$

\section{Lagrange  functions and first Bernstein inequalities}\label{Lagrange_and_Bernstein}

In this section we investigate some further results about the RBF $k_m$; 
namely, we consider analytic properties of the Lagrange functions.
These have been presented in \cite{HNW}, but we explain them below for the sake of completeness.

After this we give a first class of Bernstein estimates, valid for linear combinations of Lagrange functions.
%%%%%%%%%%
%
%
%
%%%%%%%%%%
\subsection{Lagrange functions}\label{SS:Lagrange}
For   a  finite $X\subset \reals^d$, there exists a family of (uniquely defined) functions 
$(\chi_{\xi})_{\xi\in X}$ satisfying $\chi_{\xi}\in S(X)$ and $\chi_{\xi}(\zeta) = \delta(\xi,\zeta)$ 
for all $\zeta\in X$. 
We may take the $\caln(k_m)$ inner product of two Lagrange functions $\chi_{\xi},\chi_{\zeta}\in S(X)$, noting
 that they have the form
$\chi_{\xi} =\sum_{\eta\in X} A_{\eta,\xi} k_m(\cdot-\eta)+p$ 
and $\chi_{\zeta} =\sum_{\eta\in X}  A_{\eta,\zeta} k_m(\cdot-\eta)+\tp$
(in the case of Mat{\'e}rn functions $k_m = \kappa_m$, we have $p = \tp = 0$),
to obtain
\begin{equation}\label{inner_product_coefficients}
\langle \chi_{\xi} , \chi_{\zeta} \rangle_{\caln(k_m)}
=\big \langle \chi_{\xi} ,  \textstyle{\sum_{\eta\in X}}A_{\eta,\zeta} k_m(\cdot-\eta)+\tp \big\rangle_{\caln(k)} 
= \sum_{\eta \in X}A_{\eta,\zeta}\chi_{\xi}(\eta)  = A_{\xi,\zeta}. 
\end{equation}

\paragraph{Lagrange function coefficients}
 We can make the following `bump estimate' 
 which uses a bump function 
 $\psi_{\xi,q} = \psi(\frac{\cdot - \xi}{q}):\reals^d\to [0,1]$  that is compactly supported in $B(\xi,q)$ and 
 satisfies $\psi_{\xi,q}(\xi)=1$ on a neighborhood of $q$. We have
\begin{equation}
\label{bump_estimate}
\|\chi_{\xi}\|_{\caln(k_m)}
\le
\|\psi_{\xi,q}\|_{\caln(k_m)}
\le
C \|\psi_{\xi,q}\|_{W_2^{m}(\reals^d)}
 \le Cq^{\frac{d}2-m}.
 \end{equation}
This follows because $\chi_{\xi}$ is the best interpolant to $\zeta\to \delta(\xi,\zeta)$. 
As a consequence, 
Lagrange coefficients are uniformly bounded: 
\begin{equation}\label{basiccoeffs}
|A_{\xi,\zeta}| = |\langle \chi_{\xi},\chi_{\zeta}\rangle_{\caln(k)}|\le C q^{d-2m}.
\end{equation}

\paragraph{Better decay} 
For the kernels considered in this article, and more generally for the framework given in \cite{HNW} and \cite{HNW-p},
 to get desired estimates for Lagrange functions
  over a compact region $\Omega\subset \reals^d$ 
the interpolatory conditions must be satisfied on a point set that is suitably dense
in a fairly large neighborhood of $\Omega$. 
To handle this, we use the quasi-uniform extension $\X$ developed in  Section \ref{bounded_domain}.
This brings us to the definition of $V_{\Xi}$.
\begin{definition}
For a compact set $\Omega\subset \reals^d$ and a finite subset $\Xi\subset \Omega$, let $\X$ be the extension
to $ \{x\in \reals^d\mid \dist(x,\Omega)\le \mathrm{diam}(\Omega)\}$ given in (\ref{finite_extension})
in Section \ref{SS:Extending_points}. Then for
the system of Lagrange functions $(\chi_{\xi})_{\xi\in\X}$ generated by $k_m$ over $\X$, let 
$$V_{\Xi} := \spam \{\chi_{\xi}\mid \xi \in \Xi\}.$$
\end{definition}
\noindent In particular $V_{\Xi}\subset S(\X)$.

For $\xi$ in the original set $\Xi$, we have the improved estimates:
\begin{equation}
\label{lagrange_decay}
\|\chi_{\xi} \|_{W_2^m(\reals^d \setminus B(\xi,R))} 
\le 
C q^{d/2 - m} \mathrm{exp}\left(-\mu \frac{R}{h}\right),\ \text{for all }0<R< \dist(\xi,\partial \Omega).
\end{equation}
This is demonstrated in Appendix \ref{Appendix_A}, specifically in Lemma \ref{appendix_matveev}.
This leads to improved estimates. For $\xi\in\Xi$ and  all $x\in \Omega$
\begin{equation}\label{ptwise}
|\chi_{\xi}(x)| %
\le C \rho^{m-d/2} \mathrm{exp}\left(-\mu \frac{|x-\xi |}{h}\right). 
\end{equation}
Likewise, for $\xi,\zeta\in \Xi$,
\begin{equation}\label{coeff}
|A_{\xi,\zeta}| 
\le 
C q^{d-2m} \mathrm{exp}\left(-{\nu} \frac{|\xi - \zeta|}{h}\right).
\end{equation}

%%%%%%%%%%%%%%%%%%%%%%%%%%%%
%
%%%%%%%%%%%%%%%%%%%%%%%%%%%%%

\subsection{Stability of the Lagrange-function basis for $V_\Xi$ on $\Omega$} 
Recall that $V_\Xi=\spn \{\chi_\xi\}_{ \xi\in \Xi}$, where $\Xi$ is a subset of all of the centers in $\X$. We begin by defining the operator $T: \comps^\Xi \to V_\Xi$ by $T\bfa = \sum_{\xi\in\Xi} a_{\xi} \chi_{\xi}=:s$. In other words, $T$ is the  \emph{synthesis} operator, which takes a set of coefficients $\{a_\xi \}_{\xi \in \Xi}$ and outputs a function $s\in V_\Xi$. Because the basis in consideration is the Lagrange basis, the coefficients satisfy $T\bfa(\xi) = s(\xi) = a_\xi$ for $\xi\in \Xi$ and therefore, for the basis $\{\chi_\xi\}_{ \xi\in \Xi}$, the operator $T$ is an interpolation operator.

If we use the $\ell_p(\Xi)$ norm for $\comps^\Xi$ and $L_p(\Omega)$ for $V_\Xi$, then the stability of the basis, relative to these norms, is measured by comparing $\|\bfa\|_{\ell_p(\Xi)}$ and $\|s\|_{L_p(\Omega)}$.
We show this with the following proposition, which indicates that if $s\in V_\Xi$ is small (relative to $L_p(\Omega)$) then its coefficients are small in $\ell_p(\Xi)$ 
(and likewise, if the coefficients of $s$ are small in $\ell_p(\Xi)$, so too is the norm of $s\in L_p(\Omega)$).

%%%%%%%%%%%%%%%%%%%%%%%
%
% Comparison
%
%%%%%%%%%%%%%%%%%%%%%%%

\begin{proposition}\label{full_lowercomparison}{\bf(Lagrange Basis Stability)}
There is a constant $h_0 = h_0(m,d)$, so that for 
$\Xi\subset \Omega$ satisfying
$h(\Xi,\Omega)<h_0$  and $1\le p \le \infty$, we have constants  $c_1= c_1(m,\rho,\Omega)>0$ and $c_2=c_2(m,\Omega)$
so that
\begin{equation}\label{full_riesz_ineq}
c_1 \left\| \bfa  \right \|_{\ell_p(\Xis)}     
\le  
q^{-d/p} \| \textstyle{\sum_{\xi\in\Xis}}\,a_{\xi}\chi_{\xi} \|_{L_p(\M)} \le c_2 \rho^{m+d/p} \left\|\bfa \right\|_{\ell_p(\Xi)}.
\end{equation}
\end{proposition}
We remark that the dependence in the lower constant $c_1$ on $\rho$  can be made explicit. This is worked out in Lemma \ref{lowercomparison}.
%%%%%%%%%%%%%%%%%%%%%%%%%
%%%%%%%%%%%%%%%%%%%%%%%%%
\begin{proof}
The proof is given in Appendix \ref{Appendix_B}. Lemma \ref{uppercomparison} provides the upper bound and  Lemma \ref{lowercomparison} gives the lower bound.
\end{proof}

Another way to think of this inequality is as an $L_p(\Omega)$ Marcinkiewicz-Zygmund (MZ) inequality. 
Such inequalities are used to relate the $L_1$ norm of a trigonometric polynomial to the 
$\ell_1$ norm of the polynomial  evaluated on some fixed, finite set. 
MZ inequalities have also been developed for spherical polynomials on $\mathbb S^d$ \cite{Mhaskar-etal-01-1}. 
For spherical polynomials in $\S^d$, there is another type of inequality, a Nikolskii inequality. 
On $\S^d$, these have the form $\|S\|_{L_p}\le CL^{d(\frac{1}{r}-\frac{1}{p})_+} \|S\|_{L_r(\S^d)}$, 
for any degree $L$ spherical polynomial. Our result below establishes such an inequality for $V_\Xi$.

\begin{corollary}\label{TPS_Nikolskii} {\bf(Nikolskii Inequality)}
With the assumptions and notation of Proposition~\ref{full_lowercomparison}, and with $1\le p,r\le \infty$, we have that 
\begin{equation}\label{TPS_Nikolskii_ineq}
\|s\|_{L_p(\Omega)} \le C
q^{-d(\frac{1}{r}-\frac{1}{p})_+} \|s\|_{L_r(\Omega)}, \ s\in V_\Xi
\end{equation}
with $C = C(m,\rho,\Omega,p,r) $.
\end{corollary}

\begin{proof}
Recall that, for $\bfa \in \comps^\Xi$, $\| \bfa \|_{\ell_p(\Xi)} \le N^{(\frac{1}{p}-\frac{1}{r})_+} \| \bfa \|_{\ell_r(\Xi)}$, where $N = \#\Xi$. Since $N \sim q^{-d}$, this inequality implies that $\| \bfa \|_{\ell_p(\Xi)} \le C_{\Omega,r,p}q^{-d(\frac{1}{p}-\frac{1}{r})_+} \| \bfa \|_{\ell_r(\Xi)}$. From this and \eqref{full_riesz_ineq}, we thus have
\[
\|s\|_{L_p(\M)}\le  C_{\rho,m,\Omega,p,r}q^{d\big(\frac{1}{p}-(\frac{1}{p}-\frac{1}{r})_+\big)} \|\bfa \|_{\ell_r(\Xi)} 
\le C_{\rho,m,\Omega,p,r} q^{d\big(\frac{1}{p}-\frac{1}{r} - (\frac{1}{p}-\frac{1}{r})_+\big)}\|s\|_{L_r(\M)}.
\]
The result follows from the identity $x - (x)_+ = -(-x)_+$.
\end{proof}

%%%%%%%%%%%%%%%%%%%%%%%%%%%%
%
%%%%%%%%%%%%%%%%%%%%%%%%%%%%%
\subsection{Bernstein type estimates for (full) Lagrange functions}\label{full_bernstein}
In this section we will provide a Bernstein (or inverse) theorem relating Sobolev norms of
functions in $V_{\Xi}$ to the corresponding $\ell_p$  norms on the coefficients. 
This is the key to controlling the Sobolev norm of the function in $V_{\Xi}$ by its  $L_p(\Omega)$  norm.

Before proceeding, we first prove two lemmas.

\begin{lemma}\label{zero_scale}
Let $1\le p\le \infty$, $s>d/2$ and $0\le \sigma \le s$. Suppose $O$ is a fixed open set with Lipschitz boundary and as before $O_R = \{x\mid x/R\in O\}$. Suppose further that $W_2^{s}(O)$ is embedded continuously in $W_p^{\sigma}(O)$ (where we take $C^{\sigma}$ in case $p=\infty$).
Then there is a constant $C$ depending on $O$, $p$, $s$ and $\sigma$ so that
if $U\in W_p^{\sigma}(O)$ and if the set $X$ of zeros of $U$ in $O$ 
are sufficiently dense that 
$h(X,O)<1$ and
Lemma \ref{zl} applies then for
 $u(x) = U(x/R)$, we have
$$|u|_{W_p^{\sigma}(O_R)} \le C R^{s-\sigma +d/p- d/2} |u|_{W_2^{s}(O_R)}$$
\end{lemma}
\begin{proof}
Lemma \ref{scaling} shows that
 $|u|_{W_p^{\sigma}(O_R)} 
 =
  R^{d/p-\sigma} |U|_{W_p^{\sigma}(O)}
$.
Because $\|U\|_{W_p^{\sigma}(O)} \le C\|U\|_{W_2^{s}(O)}$,
we have that 
$ |u|_{W_p^{\sigma}(O_R)} \le C  R^{d/p-\sigma} \|U\|_{W_2^{s}(O)}
=
C   R^{d/p-\sigma} 
  \left(\sum_{j=0}^k \begin{pmatrix} k\\ j\end{pmatrix}|U|_{W_2^j(O)}^2+ |U|_{W_2^{s}(O)}^2\right)^{1/2}$ 
  (in case $s$ is fractional; we leave the necessary modification to the reader in case $s\in\N$).
  We  now apply the zeros estimate \cite[Theorem 4.2]{NWW_Bernstein} to each term on the right hand side, 
  obtaining
$$ |u|_{W_p^{\sigma}(O_R)} 
=
C   R^{d/p-\sigma} 
 |U|_{W_2^{s}(O)},$$
 since  $|U|_{W_2^j(O)}^2\le C h^{2(s-j)} |U|_{W_2^{s}(O)}^2$ for all $j$.
Applying Lemma \ref{scaling} again,
(this time with $p=2$), 
yields the desired estimate.
\end{proof}

\begin{lemma}\label{sobolev2discrete_bern_single ball}
Suppose  $\Xi$ is sufficiently dense (meaning $h(\Xi,\Omega)\le h_0$ for a constant $h_0 = h_0(d,m)>0$)
and $\eta \in \Xi$. Decompose $\Xi$ into disjoint annuli $\Xi = \bigcup_{j=0}^{\infty}\Xi_j(\eta) $: where
$\Xi_j(\eta) := \{\zeta\in\Xi\mid 2^{j-1} h \le \dist(\zeta,\eta) \le 2^{j} h\}$ for $j>0$ and 
$\Xi_{0}(\eta):= \{\zeta\in\Xi\mid 0\le \dist(\zeta,\eta) \le h\}$.

We have, for $2\le p<\infty$ and
$0\le \sigma \le m -d/2+d/p$ 
that
\begin{equation}\label{LP-Synthesis}
\big\|\textstyle{\sum_{\xi\in \Xi}} a_{\xi} \chi_{\xi}\big\|_{W_p^{\sigma}(B(\eta,3h))}^p
\le
C \rho^{p(m+d/2)-d} 
h^{d-p\sigma }  
\sum_{j=0}^{\infty} 
     2^{j(d+1)(p-1)}e^{-{\nu} p2^{j-1}}
  \sum_{\xi\in \Xi_j(\eta)}|a_{\xi}|^p
\end{equation} 
with $C= C(p,\sigma, m,d)$ and $\nu = \nu(m,d)$.
\end{lemma}
\begin{proof}
Repeatedly applying the quasi-triangle inequality $\|a+b\|^p \le 2^{p-1}( \|a\|^p +\|b\|^p)$  to this sum gives
$$
\big\|\sum_{\xi\in \Xi} a_{\xi} \chi_{\xi}\big\|_{W_p^{\sigma}(B(\eta,3h))}^p
\le 
 \sum_{j=0}^{\infty} 2^{(j+1)(p-1)} 
\big\|\sum_{\xi\in \Xi_j(\eta)} a_{\xi} \chi_{\xi}(x)\big\|_{W_p^{\sigma}(B(\eta,3h))}^p.
$$
Observe that $\#\Xi_j(\eta)  \le\omega_d \rho^d 2^{jd}$ 
(where the constant $\omega_d$ depends on $d$), 
so the generalization of the above quasi-triangle inequality 
$\|\sum_{j=1}^na_j\|^p \le n^{p-1}\sum_{j=1}^n \|a_j\|^p $ 
gives
\begin{equation}\label{quasi-triangle}
\big\|\sum_{\xi\in \Xi} a_{\xi} \chi_{\xi}\big\|_{W_p^{\sigma}(B(\eta,3h))}^p
\le 
 \sum_{j=0}^{\infty} 
    2^{(j+1)(p-1)} 
    (\omega_d
    \rho^{d}
    2^{jd})^{p-1}
    \sum_{\xi\in \Xi_j(\eta)}
      |a_{\xi}|^p \left\| \chi_{\xi}\right\|_{W_p^{\sigma}(B(\eta,3h))}^p.
\end{equation}
For $\dist(\xi,\eta)=R$ sufficiently large, 
we have the inclusion $B(\eta,3h)\subset \Om \setminus B(\xi,R-3h) $.
Applying the zeros lemma  \cite[Theorem 1.1]{NWW}  gives
$$
\left\| \chi_{\xi} \right\|_{W_p^{\sigma}(B(\eta,3h))} 
\le 
\left\| \chi_{\xi} \right\|_{W_p^{\sigma}(\Om \setminus B(\xi,R-3h))}
\le 
C h^{m-\sigma - (d/2 -d/p)_+} \left\| \chi_{\xi}\right\|_{W_2^{m}(\Om\setminus B(\xi,R-3h))}.
$$
Applying the energy estimate (\ref{lagrange_decay}) and noting that $d/2 - d/p\ge 0$
 gives us
$$\left\| \chi_{\xi} \right\|_{W_p^{\sigma}(B(\eta,3h))} \le C h^{m-\sigma - (d/2 -d/p)_+} q^{d/2 -m} e^{-\nu R/h}=
C \rho^{m-d/2} h^{d/p-\sigma }  e^{-\nu R/h}
.$$
We note that for $\xi$ in the annular set $ \Xi_j(\eta)$, $\dist(\xi,\eta)=R\ge h2^{j-1}$, so
$e^{-\nu R/h}\le   e^{-\frac{\nu}{2} 2^{j}}$.
Applying this to (\ref{quasi-triangle}) gives
\begin{eqnarray}
\big\|\sum_{\xi\in \Xi} a_{\xi} \chi_{\xi}\big\|_{W_p^{\sigma}(B(\eta,3h))}^p
&\le& 
C \omega_d^{p-1}\rho^{p(m-d/2)+d(p-1)} 
\sum_{j=0}^{\infty} 
    2^{j(d+1)(p-1)}\sum_{\xi\in \Xi_j(\eta)}|a_{\xi}|^p
h^{d-p\sigma}  e^{-\frac{\nu}{2} p2^{j}}\nonumber\\
&\le& 
C \rho^{p(m+d/2)-d} 
h^{d-p\sigma }  
\sum_{j=0}^{\infty} 
     2^{j(d+1)(p-1)}e^{-\frac{\nu}{2} p2^{j}}
    \sum_{\xi\in \Xi_j(\eta)}|a_{\xi}|^p
\end{eqnarray}
\end{proof}

Note that when $p=\infty$, we use
only integer smoothness $\sigma =k \in \ints$ and  the standard space $C^k(\overline{\Omega})$ of $k$ times integral functions over $\overline{\Omega}$.
%%%%%
%
%
%
%%%%%
\begin{theorem}\label{sobolev2discrete_bern}
For a sufficiently dense set $\Xi$ (meaning $h(\Xi,\Omega)\le h_0$ for a constant $h_0 = h_0(d,m)>0$)
we have, for 
$0\le \sigma \le m -(d/2-d/p)_+$ when $1\le p< \infty$ (or $\sigma \in \nats$ with  $0\le \sigma<m - d/2$ if $p=\infty$),
\begin{equation}\label{LP-Synthesis}
\big\|\textstyle{\sum_{\xi\in \Xi}} a_{\xi} \chi_{\xi}\big\|_{W_p^{\sigma}(\Omega)} 
\le 
C \rho^{m+d/2+d/p} h^{d/p -\sigma } \left\|\bfa \right\|_{\ell_p(\Xi)}
\end{equation} 
with $C= C(p,\sigma, m,d)$.
\end{theorem}
%%%%%
%
%
%
%%%%%
\begin{proof}
This is handled in four cases: $p=\infty$, $2\le p<\infty$, $p=1$ and $1 < p<2$.

\paragraph{Case 1: $p=\infty$}
If $\sigma\in\ints$, we simply need to bound 
$\max_{|\alpha|=\sigma}\max_{x\in\Omega} \sum_{\xi\in \X} |D^{\alpha}\chi_{\xi}(x)|_{\infty}$. 
To do this, consider a point $x\in\Omega$ and a ball 
$B(x,r)\subset \Om$ with $r = h\max(16 m^2, 1/h_1)$ and $h_1$ is the constant from
the zeros lemma \ref{zl}. 
In this case, we use a Bramble-Hilbert argument involving the averaged Taylor polynomial 
$Q^m \chi_{\xi}$ of degree $m-1$ described in Brenner-Scott \cite{Bren}. 
It follows from \cite[(3.9)]{HNW}
that 
$\|D^{\alpha}Q^m \chi_{\xi}\|_{L_{\infty}(B(x,r))} \le C r^{m-|\alpha|-d/2} |\chi_{\xi}|_{W_2^m(B(x,r))} $.
Likewise, we can estimate 
$\|D^{\alpha}(\chi_{\xi} - Q^m \chi_{\xi})\|_{L_{\infty}(B(x,r))}$ 
by first using Lemma \ref{scaling} (with $U((y-\xi)/r))  = \chi_{\xi}(y) - Q^m \chi_{\xi}(y)$ 
and the  embedding
 $W_2^{m-|\alpha|}(B(0,1))\subset C(B(0,1))$ 
 to obtain
 $$|(\chi_{\xi} - Q^m \chi_{\xi})|_{C^{|\alpha|}(B(x,r))}
 \le C r^{-|\alpha|}|U|_{C^{|\alpha|}(B(0,1))}
 \le C r^{-|\alpha|}|U|_{W_2^m(B(0,1))}$$
 Rescaling, 
 gives the estimate 
$$\|D^{\alpha}(\chi_{\xi} - Q^m \chi_{\xi})\|_{L_{\infty}(B(x,r))}
\le 
C r^{-|\alpha|}
\left(\sum_{j=0}^{m} r^{2j-d}\begin{pmatrix}m\\j\end{pmatrix} |\chi_{\xi} - Q^m \chi_{\xi}|_{W_2^{j}(B(x,r))}^2\right).
$$
Each seminorm in this last expression can be estimated by the Bramble-Hilbert Lemma, allowing us to 
 bound the above by
$ C r^{m-|\alpha|-d/2} |\chi_{\xi}|_{W_2^{m}(B(x,r))}.$
Together with the estimate on $D^{|\alpha|}Q^m \chi_{\xi}$, and recalling that
$r = K h$, we have
$$
|D^{\alpha} \chi_{\xi} (x)| \le \|D^{\alpha}Q^m \chi_{\xi}\|_{L_{\infty}(B(x,r))} + \|D^{\alpha}(\chi_{\xi} - Q^m \chi_{\xi})\|_{L_{\infty}(B(x,r))}\le C h^{m-|\alpha|-d/2} |\chi_{\xi}|_{W_2^{m}(B(x,r))}.
$$
From here, we apply the energy estimate (\ref{lagrange_decay}) to obtain 
\begin{equation}
\label{derivative_lagrange_decay}
|D^{\alpha}\chi_{\xi}(x)|
\le 
C h^{m-|\alpha|-\tfrac{d}2} q^{\tfrac{d}2-m} e^{-\nu \frac{|x-\xi|}{h}}
=
C \rho^{m-\tfrac{d}2} h^{-|\alpha|}  e^{-\nu \frac{|x-\xi|}{h}}.
\end{equation}

The sum over $\Xi$ can be carried out over annular regions 
$\Xi_j(x) = \{\xi\in \Xi \mid jh\le \dist(\xi,x)< (j+1)h\}$. 
This leaves
$$ 
 \sum_{j=0}^{\infty} \sum_{\xi\in \Xi_j(x)} |D^{\alpha}\chi_{\xi}(x)|
\le 
C \rho^{m-\tfrac{d}2} h^{-|\alpha|}   \sum_{j=0}^{\infty} \sum_{\xi\in \Xi_j(x)}  e^{-\nu \frac{|x-\xi|}{h}}
\le 
C \rho^{m-\tfrac{d}2} h^{-|\alpha|}   \sum_{j=0}^{\infty} \rho^d (j+1)^d e^{-\nu j}
\le 
C\rho^{m+\tfrac{d}2} h^{-|\alpha|}.
$$
In the last inequality, we use the fact that the sum $ \sum_{j=0}^{\infty} (j+1)^d e^{-\nu j} =C$
depends on $d$ and $m$ (but not $\rho$).
\paragraph{Case 2: $2\le p<\infty$}
We treat this case in two stages. At first, we consider $\sigma=k\in \N$, treating fractional Sobolev exponents for later. 

\subparagraph{Case 2i: $\sigma=k\in\N$.} By subadditivity, the Sobolev norm may be taken over overlapping balls 
$$
\big\|\sum_{\xi\in \Xi} a_{\xi} \chi_{\xi}\big\|_{W_p^{k}(\Omega)}^p
\le 
\sum_{\eta\in\Xi}   \big\|\sum_{\xi\in \Xi} a_{\xi} \chi_{\xi}(x)\big\|_{W_p^{k}(B(\eta,h))}^p.
$$
Applying Lemma \ref{sobolev2discrete_bern_single ball} to the norm over each ball $B(\eta,h)\subset B(\eta,3h)$
gives 
\begin{equation*}
\big\|\sum_{\xi\in \Xi} a_{\xi} \chi_{\xi}\big\|_{W_p^{k}(\Omega)}^p
\le
C \rho^{pm+pd/2-d} 
h^{d-p\sigma }  
\sum_{j=0}^{\infty} 
     2^{j(d+1)(p-1)}e^{-{\nu} p2^{j-1}}
\sum_{\eta\in\Xi}  
  \sum_{\xi\in \Xi_j(\eta)}|a_{\xi}|^p
\end{equation*} 
We may exchange  summation between $\xi$ and $\eta$, 
noting that 
$\eta\in\Xi_j(\xi)$ iff $\xi\in\Xi_j(\eta)$. This implies the estimate
$\sum_{\eta\in\Xi}  \sum_{\xi\in \Xi_j(\eta)}|a_\xi|^p 
= 
\sum_{\xi\in \Xi}\sum_{\eta\in \Xi_j(\xi)}|a_\xi|^p\le \omega_d \rho^d 2^{jd}\sum_{\xi\in \Xi}|a_\xi|^p$.
 Consequently, 
\begin{equation}\label{Case2_integer_case}
\big\|\sum_{\xi\in \Xi} a_{\xi} \chi_{\xi}\big\|_{W_p^{k}(\Omega)}^p
\le 
C \omega_d^{p-1}\rho^{p(m+d/2)} h^{-pk +d}  
\bigg( \sum_{j=0}^{\infty} 2^{j(d+1)p}e^{-\frac{\nu}{2} p2^{j}}\bigg)\sum_{\xi\in \Xi}|a_{\xi}|^p.
\end{equation}
The result follows by summing the series and taking the $p$th root.

\subparagraph{Case 2ii: $\sigma\notin\N$.}Let $\sigma = k+\delta$ with $0<\delta<1$, and  employ Lemma \ref{sa}, 
using the neighborhoods $\{B(\eta,h)\cap \Omega \mid \eta \in \Xi\}$ as $\{\vt_j\mid j\in\cN\}$
and $\{B(\eta,3h)\cap \Omega \mid \eta\in \Xi\}$ for $\{v_j\mid j \in \cN\}$. Note that for this choice of cover, $M \le C \rho^d$.
Indeed, 
for any $x\in \Omega$, enumerate the centers $\{\xi \in \Xi \mid \dist(\xi,x)<h\}$  as $\xi_1,\dots, \xi_{n}$. Then
$x\in B(\xi_j,h)$ for each $j=1\dots n$. 
Because the  balls $B(\xi_j,q) $ are disjoint, 
it follows that
$n ( c_{d,1} q^d)  = \mathrm{vol}( \bigcup_{j=1}^n B(x_j,q) \cap B(x,h))\le c_{d,2} h^d$, so $n   \le C_d (h/q)^d$.
(Here $c_{d,2}h^d$ is the volume of the ball of radius $h$, and $c_{d,1}q^d$ is the volume of the
portion in $B(x,h)$ of any ball which is centered in $B(x,h)$.)
This guarantees that 
$$
\big\|\sum_{\xi\in \Xi} a_{\xi} \chi_{\xi}\big\|_{W_p^{\sigma}(\Omega)}^p
\le 
\sum_{\eta\in\Xi}   \big\|\sum_{\xi\in \Xi} a_{\xi} \chi_{\xi}\big\|_{W_p^{\sigma}(B(\eta,3h))}^p
+
C\rho^d h^{-p\delta} \|\sum_{\xi\in \Xi} a_{\xi} \chi_{\xi}\big\|_{W_p^{k}(\Omega)}^p
.$$

We apply (\ref{Case2_integer_case}) to bound the second term, which gives 
$$
C\rho^d h^{-p\delta} \|\sum_{\xi\in \Xi} a_{\xi} \chi_{\xi}\big\|_{W_p^{k}(\Omega)}^p \le 
C \rho^{p(m+d/2)+d} h^{-pk +d - \delta} |a_{\xi}|^p.
$$
The first term is handled precisely as the integer case $\sigma=k$, which has  been discussed above,
leaving 
$$\sum_{\eta\in\Xi}   \big\|\sum_{\xi\in \Xi} a_{\xi} \chi_{\xi}\big\|_{W_p^{\sigma}(B(\eta,3h))}^p \le 
C \rho^{p(m+d/2)} h^{d/p -\sigma } \left\|\bfa \right\|_{\ell_p(\Xi)}.
$$
\paragraph{Case 3: $p=1$}
We again consider the proof in two steps, first for the  case of integer smoothness, 
where the Sobolev norm is sub-additive on sets, and then in the fractional case, 
where we can apply  Lemma \ref{sa}.\footnote{In this case, because $0\le \sigma\le m$, 
one could just as easily adopt the strategy of proving the result for the extrema 
$\sigma=0$ and $\sigma=m$, and then using interpolation of operators to bound the  synthesis 
operator $T: \ell_1(\Xi) \to W_1^{\sigma}(\Omega)$, noting that $W_1^{\sigma}(\Omega)$ is the  
$(\sigma/m,1)$ interpolation space between $L_1(\Omega)$ and $W_1^m(\Omega)$.}

As an initial simplification, note that the triangle inequality gives
$
\|s\|_{W_1^{\sigma}(\Omega)}
\le 
\|\bfa \|_{\ell_1(\Xi)} (\max_{\xi\in\Xi}\|\chi_{\xi}\|_{W_1^{\sigma}(\Omega)})
$, so we need only to consider the size of $\|\chi_{\xi}\|_{W_1^{\sigma}(\Omega)}$.

\subparagraph{Case 3i: $\sigma=k\in\N$.} We proceed, as in Case 2, by first considering $\sigma=k\in \N$ and using subadditivity of the norm.
For any integer $K$, we have 
$
\|\chi_{\xi}\|_{W_1^{k}(\Omega)}
\le 
\|\chi_{\xi}\|_{W_1^{k}(B(\xi,Kh))} + \|\chi_{\xi}\|_{W_1^{k}(\Omega\setminus B(\xi,Kh))}
$.

The first term satisfies
$\|\chi_{\xi}\|_{W_1^{k}(B(\xi,Kh))}
\le 
\omega_d (Kh)^{d/2}  \|\chi_{\xi}\|_{W_2^{k}(B(\xi,Kh))}
$. 
For $K$ sufficiently large (a constant depending only on $d$), the zeros estimate \cite[Theorem 4.2]{NWW_Bernstein} gives  
\begin{equation} \label{Case3_integer_case_inside}
\|\chi_{\xi}\|_{W_1^{k}(B(\xi,Kh))}
\le 
\omega_d (Kh)^{d/2}  h^{m-k}\|\chi_{\xi}\|_{W_2^{m}(\Om)}
\le 
C K^{d/2} \rho^{m-d/2}h^{d-k}
\end{equation}

The second term may be controlled by decomposing 
$\Omega\setminus B(\xi,Kh) = \bigcup_{\ell=K}^{\infty} A_\ell$ 
en annuli
(i.e., by taking $A_\ell := \{x\in \Omega\mid \ell h\le \dist(x,\xi)\le (\ell+1)h\}$). Subadditivity gives
\begin{eqnarray*}
\|\chi_{\xi}\|_{W_1^{k}(\Omega\setminus B(\xi,Kh))}
&\le &
\sum_{\ell=K}^{\infty} \|\chi_{\xi}\|_{W_1^{k}(A_\ell)}\\
&\le &
\sum_{\ell=K}^{\infty} (\mathrm{vol}(A_\ell))^{1/2}\|\chi_{\xi}\|_{W_2^{k}(A_\ell)}\\
&\le &
\sum_{\ell=K}^{\infty} C ((\ell+1)h)^{d/2}   h^{m-\sigma}  \|\chi_{\xi}\|_{W_2^{m}(A_\ell)}
\end{eqnarray*}
In the final line we have applied the zeros estimate 
(and simultaneously estimated the volume of the annulus $A_\ell$).
At this point, we can apply the energy estimate  (\ref{lagrange_decay}) to obtain
\begin{equation}\label{Case3_integer_case_outside}
\|\chi_{\xi}\|_{W_1^{k}(\Omega\setminus B(\xi,Kh))}
\le
\sum_{\ell=K}^{\infty} C ((\ell+1)h)^{d/2}  h^{m-\sigma} q^{d/2-m} e^{-\nu \ell}
\le
C \rho^{m-d/2} h^{d-\sigma}.
\end{equation}
Combining (\ref{Case3_integer_case_outside}) with 
(\ref{Case3_integer_case_inside}), gives the desired result for $\sigma =k\in \N$. 

\subparagraph{Case 3ii: $\sigma\notin\N$.}  
To handle the fractional case $\sigma = k+\delta$, 
we apply Lemma \ref{sa}, 
with an initial decomposition $\tilde{v}_1 = B(\xi,Kh)$, $\tilde{v}_2 = \Omega\setminus B(\xi,Kh)$,
 $v_1 = B(\xi,(K+1)h)$ and $v_2 = \Omega\setminus B(\xi,(K-1)h)$.
 Observe that these are disjoint, so the overlap constant is $M=1$. Thus we have
$$|\chi_{\xi}|_{W_1^{\sigma}(\Omega)}
\le 
|\chi_{\xi}|_{W_1^{\sigma}(B(\xi,(K+1)h))} + 
|\chi_{\xi}|_{W_1^{\sigma}(\Omega\setminus B(\xi,(K-1)h))} + 
C h^{-\delta}\|\chi_{\xi}\|_{W_1^{k}(\Omega)}.
$$
We can further decompose the middle term $|\chi_{\xi}|_{W_1^{\sigma}(\Omega\setminus B(\xi,(K-1)h))} $
en annuli 
by applying Lemma \ref{sa} a second time.
This time, we let 
$\tilde{v}_{\ell} :=  
\{x\in \Omega\mid 2^\ell (K-1) h\le \dist(x,\xi) < 2^{\ell+1}(K-1)h\}$ for $\ell = 0, 1, \dots$.
The annuli $\{\tilde{v}_{\ell}\mid \ell\in \N\}$ partition
$\Omega\setminus B(\xi,(K-1)h)$, so the overlap constant $M$ is $M=1$; in fact,
we need only the first $\ell_0 = 1+ \log_2(\diam(\Omega)/(K-1)h)$ annuli.

Define the  neighborhoods of $\vt_{\ell}$ as  $v_{\ell}  :=  \{x\in \Omega\mid 2^{\ell-1}(K-1) h\le \dist(x,\xi) < 2^{\ell+2}(K-1)h\}$. 
The sets $w_{\ell} = \Omega \setminus v_{\ell}$ 
satisfy $\dist(\tilde{v}_{\ell},w_{\ell})\ge \frac12(K-1)h \ge h$. 
Lemma \ref{sa} shows that
$$
 |\chi_{\xi}|_{W_1^{\sigma}(\Omega\setminus B(\xi,(K-1)h))}
 \le 
\left(\sum_{\ell=0}^{\ell_0}  |\chi_{\xi}|_{W_1^{\sigma}(v_\ell)}\right)
+ C h^{-\delta}\|\chi_{\xi}\|_{W_1^{k}(\Omega)}.
$$
Since $\|\chi_{\xi}\|_{W_1^{\sigma}(\Omega)} = \|\chi_{\xi}\|_{W_1^{k}(\Omega)}+|\chi_{\xi}|_{W_1^{\sigma}(\Omega)}$
and $h<1$, this leaves
\begin{equation}\label{Case3_full}
\|\chi_{\xi}\|_{W_1^{\sigma}(\Omega)}
\le 
|\chi_{\xi}|_{W_1^{\sigma}(B(\xi,(K+1)h))} 
+ 
\left(\sum_{\ell=0}^{\ell_0}   |\chi_{\xi}|_{W_1^{\sigma}(v_\ell)}\right)
 + 
C h^{-\delta}\|\chi_{\xi}\|_{W_1^{k}(\Omega)},
\end{equation}
which we must estimate.

{\em Estimating the third term in (\ref{Case3_full}):}
The final term is easiest to control: {\bf Case 3i} gives the estimate
\begin{equation}\label{Case3_thirdpart}
h^{-\delta}\|\chi_{\xi}\|_{W_1^{k}(\Omega)}\le  C\rho^{m-d/2} h^{d-k-\delta}.
\end{equation}

{\em Estimating the first term in (\ref{Case3_full}):}
The first term in (\ref{Case3_full}) is controlled in a similar way to (\ref{Case3_integer_case_inside}).
Begin by setting $R:= (K+1)h$, $u= \chi_{\xi}(\cdot - \xi)$ and $U = u(R(\cdot))$. 
Applying Lemma \ref{zero_scale} with $O= B(0,1)$ gives 
$$|\chi_{\xi}|_{W_1^{\sigma}(B(\xi,(K+1)h))}  = |u|_{W_1^{\sigma}(B(0,R)} \le 
C R^{d/2+m-\sigma} |u|_{W_2^{m}(B(0,R))} =C   h^{d/2+m-\sigma}|\chi_{\xi}|_{W_2^m(B(\xi,(K+1)h) )}
$$
The bump estimate (\ref{bump_estimate}) then gives
\begin{equation}\label{Case3_firstpart}
 |\chi_{\xi}|_{W_1^{\sigma}(B(\xi,(K+1)h))}
 \le
 C   h^{d/2+m-\sigma}|\chi_{\xi}|_{W_2^m(\Om )}
 \le
C h^{d/2+m-\sigma}q^{d/2-m}
\le
C
\rho^{m-d/2}h^{d-\sigma}.
\end{equation}

{\em Estimating the middle term in (\ref{Case3_full}):}
To handle the series appearing in (\ref{Case3_full}),
we proceed as in the last paragraph, 
 applying,  for each $\ell$, Lemma  \ref{zero_scale} , 
now with $u = \chi_{\xi}(\cdot - \xi)$, $R = 2^{\ell+2} (K-1)h$
and $U = u(R\cdot)$. In this case $O = B(0,1)\setminus B(0,1/8)$. The scaling lemma uses the embedding $W_2^{\sigma}(O)\subset W_1^{\sigma}(O)$ which incurs an embedding
constant $C$ which is independent of $\ell$.
$$
 |\chi_{\xi}|_{W_1^{\sigma}(v_\ell)} 
 =
  |u|_{W_1^{\sigma}(O_R)}
 \le
 C R^{d/2+m-\sigma}
 |u|_ {W_2^{m}(O_R)}
 =
 C (2^{\ell+2}(K-1))^{d/2+m-\sigma} h^{d/2+m-\sigma}
 \|\chi_{\xi}\|_{W_2^m(v_{\ell})}.
$$
Since $v_{\ell}$ is contained in $ \Om\setminus B(\xi, 2^{\ell-1}(K-1)h)$,
we have 
$ |\chi_{\xi}|_{W_1^{\sigma}(v_\ell)} \le C 2^{\ell(d/2+m-\sigma)} h^{d/2+m-\sigma}   \|\chi_{\xi}\|_{W_2^m(\Om\setminus B(v_{\ell})}$. 
Now we apply the 
energy estimate (\ref{lagrange_decay})
which gives 
$$   |\chi_{\xi}|_{W_1^{\sigma}(v_\ell)} \le C  2^{\ell(d/2+m-\sigma)} h^{d/2+m-\sigma}  q^{d/2-m} e^{-\mu (K-1)2^{\ell-1}}.$$ 
Observing that the infinite series $\sum_{\ell=0}^{\infty}(2^{\ell(d/2+m-\sigma)}e^{-\mu (K-1)2^{\ell-1}}$
 converges to a constant depending only on $d$ and $m$, 
we can bound the middle term:
\begin{equation}\label{Case3_secondpart}
\left(\sum_{\ell=0}^{\ell_0}   |\chi_{\xi}|_{W_1^{\sigma}(v_\ell)}\right)
 \le 
C   
h^{m+d/2-\sigma} q^{d/2-m}
= 
C \rho^{m-d/2} h^{d-\sigma}.
\end{equation}

The case $p=1$ follows from the estimates (\ref{Case3_firstpart}), (\ref{Case3_secondpart}), (\ref{Case3_thirdpart}) and (\ref{Case3_full}).

\paragraph{Case 4: $1<p<2$}
In this case, we use Riesz-Thorin to estimate the norm of the  operator 
$T:\ell_p(\Xi)\to W_p^{\sigma}(\Omega)$, 
where $T$ is the synthesis operator $T\bfa = \sum_{\xi\in\Xi} a_{\xi} \chi_{\xi}$ (i.e., the linear map which takes coordinate space $\comps^{\Xi}$ into the 
vector space $V_\Xi$).
 Letting 
 $\theta = 2(\frac1p - \frac12)$ 
 (so that $\frac1p = \theta 1 + (1-\theta) \frac12$) gives
 \begin{eqnarray*}
 \|\sum_{\xi\in\Xi} a_{\xi} \chi_{\xi}\|_{W_p^{\sigma}(\Omega)}
 &\le&
 \left(C \rho^{m-d/2} h^{d-\sigma} \right)^{\theta}
  \left(C \rho^{m-d/2} h^{d/2-\sigma}\right)^{1-\theta} \|\bfa\|_{\ell_p(\Xi)}\\
   &\le&
C \rho^{m-d/2} h^{d/p-\sigma} \|\bfa\|_{\ell_p(\Xi)}.
 \end{eqnarray*}
\end{proof}
%%%%%
%
%
%
%%%%%

Using Proposition~\ref{full_lowercomparison}, we may replace the discrete norm $\|\a\|_{\ell_p(\Xi)}$
by its equivalent $h^{-d/p}\|s\|_{L_p}$, and so obtain an $L_p$ version of 
Theorem~\ref{sobolev2discrete_bern}.

%%%%%%%%%%%%
%
%
%
%%%%%%%%%%%%
\begin{corollary}\label{sobolev2continuous_bern}
With the assumptions of Theorem~\ref{sobolev2discrete_bern}, we have
\begin{equation}\label{Wp2Lp}
\big\|\textstyle{\sum_{\xi\in \Xi}} a_{\xi} \chi_{\xi}\big\|_{W_p^{\sigma}(\Omega)} 
\le 
C h^{-\sigma } \big\|\textstyle{\sum_{\xi\in \Xi}} a_{\xi} \chi_{\xi}\big\|_{L_p(\Omega)} 
\end{equation} 
with $C=C(p,\sigma, m,\rho,\Omega)$.
\end{corollary}
Explicit dependence of $C$ on $\rho$ can be obtained from (\ref{LP-Synthesis}) and Lemma B.6.

%%%%%%%%%%%%%%%%%%%%%%%%%%%%%%%
%%%%%%%%%%%%%%%%
%%%%%%%%%%%%%%%%
%%%%%%%%%%%%%%%%
%%%%%%%%%%%%%%%%%%%%%%%%%%%%%%%
\section{Local Lagrange functions}
We now consider a new class of  functions $b_{\xi}\in S( \X)$, $\xi \in \Xi$,  constructed in a local and cost-effective way, 
employing only a small set of centers in $\X$ that are near $\xi$. For each $\xi\in \Xi$, this small set is called the local {\em footprint} of $\xi$ and denoted by $ \Upsilon(\xi)\subset \X$ (see Definition~\ref{local_l_space_def}). Each $b_\xi$ is a Lagrange interpolant, centered at $\xi$, for points in $\Upsilon(\xi)$. The set $\Upsilon(\xi)$ is chosen to give  $b_{\xi}$ fast decay away from $\xi$, although not the exponential. The size of the footprint is controlled by a parameter $K>0$.

Unlike the {\em full} Lagrange functions, the local versions do not satisfy interpolatory conditions throughout $\X$. There is no guarantee that they will have zeros outside of the set $ \Upsilon(\xi)$ - as a result
the operator $T\bfa = \sum_{\xi\in \Xi} a_{\xi} b_{\xi}$ does not satisfy $T\bfa (\xi) = a_{\xi}$. It is only a {\em quasi-interpolant} (approximation on the sphere with this operator was considered in \cite{FHNWW}).

As in \cite{FHNWW} the analysis of this new basis 
is considered in two steps. 
First, an intermediate basis function
$\wch_{\xi}$
 is constructed and studied: the {\em truncated Lagrange function}. 
 These functions employ the same footprint as $b_{\xi}$ 
 (i.e., they are members of $S(\Upsilon(\xi))$)
but their construction is global rather than local. 
This topic  is considered in Section \ref{SS:TLF}. Then, a comparison is made between the 
truncated Lagrange function and the local Lagrange function. 
The error between local and truncated Lagrange functions is controlled by the size of the  coefficients  in the representation of
 $b_{\xi} -\wch_{\xi}$ using the standard (kernel) basis for $S(\Upsilon(\xi))$.
This is considered  in Section \ref{SS:LLF}.

\subsection{Footprint and local Lagrange function}
\begin{definition}\label{local_l_space_def}
For a compact set $\Omega\subset \reals^d$ and a finite subset $\Xi\subset \Omega$, let $\X$ be the extension
to $ \{x\in \reals^d\mid \dist(x,\Omega)\le \mathrm{diam}(\Omega)\}$ given in Section \ref{SS:Extending_points}. 
For a positive parameter $K$, define $\Upsilon(\xi):= \{\zeta\in \X\mid |\xi-\zeta| \le Kh |\log h|\}$ for each $\xi\in\Xi$.
Then for the system of local Lagrange functions $(b_{\xi})_{\xi\in\X}$, where each $b_{\xi}$ is the Lagrange function centered at $\xi$, generated by $k_m$ over  $\Upsilon(\xi)$, let
$$\V_{\Xi} := \spam\{b_{\xi}\mid \xi \in \Xi\}.$$
\end{definition}

Note in particular that $\V_{\Xi}\subset S(\X)$. Indeed, it is contained in
a slight expansion of $S(\Xi)$. Namely,
 $\V_{\Xi}\subset S(\varUpsilon)$,
where $\varUpsilon:=  \bigcup_{\xi\in\Xi} \Upsilon(\xi) \subset \{\xi \in \X\mid \dist(\xi,\Omega)\le K h|\log h|\}$. 

The construction of each $b_{\xi}$ depends only on its nearby neighbors in $\Upsilon(\xi)$, 
so the majority of points in $\X$ are unnecessary from a computational point of view. However, 
the (analytic properties of) full Lagrange functions $\chi_{\xi}$ generated
by $k_m$ over $\X$ will still be of use in proving theorems,  so we will continue to refer to the extended set $\X$, 
even though  much of it plays no role in the construction of the functions $b_{\xi}$.

In our main result, we make use of the following:

\begin{quotation}
Let $(\chi_{\xi})_{\xi\in\Xi}$ be the family of ``full'' Lagrange functions constructed by $k_m$ using the extended
point set $\X$. For any $J>0$,  the family $(b_{\xi})_{\xi\in\Xi}$ satisfies 
\begin{equation}\label{compatible}
\|\chi_{\xi} - b_{\xi}\|_{W_p^{\sigma}(\Omega)}\lesssim h^J, \ \text{for all }\xi \in \Xi.
\end{equation}
\end{quotation}
To obtain this result, we show that for a given $J$ there is a  $K>0$, which governs the size of the footprint, ensuring that $\|\chi_{\xi} - b_{\xi}\|_{\infty} =\mathcal{O}(h^J)$ holds. The value of $K$ depends linearly on $J$, as well as some fixed constants involving 
$m$ and $d$. %

In the following two sections, we show that  this result holds for Mat{\'e}rn (in Lemma \ref{Matern_compare}) and surface spline radial basis functions (in Lemma \ref{ss_comparison}). 
Specifically, this holds for any prescribed value of $J$, where $J$ depends linearly on $K$, as given in Definition~\ref{local_l_space_def}.

\subsection{Intermediate construction: Truncated Lagrange functions}\label{SS:TLF}

For a (full) Lagrange function 
$\chi_{\xi} = \sum_{\zeta\in \X} A_{\xi,\zeta} k(\cdot, \zeta) +p \in S( \X)$
on the point set $ \X$, the truncated Lagrange function 
$\wch_{\xi} :=  \sum_{\zeta\in\Upsilon(\xi)} \wA_{\xi,\zeta} k(\cdot, \zeta) +p$
is a function in $S(\Upsilon(\xi))$ obtained by omitting the coefficients outside of $\Upsilon(\xi)$
and  slightly modifying the remaining coefficients $\bfA_{\xi} = (A_{\xi,\zeta})$. 
(For positive definite kernels, no modification is necessary, and the construction is quite simple.)

The cost of truncating can be measured using the norm of the omitted coefficients (the tail).

\begin{lemma} \label{trunc_coeffs}
Let $K> (4m-2d)/\nu$ and for each $\xi\in \Xi$, 
let $\Upsilon(\xi) = \{\zeta\in  \X\mid |\xi-\zeta| \le K h|\log h|\}$. 
Then
$$
\sum_{\zeta\in  \X\setminus \Upsilon(\xi)} |A_{\xi,\zeta}|
\le 
C \rho^{2m} h^{K\nu/2 +d-2m}
$$
with $C = C(m,d)$.
\end{lemma}
\begin{proof}
The inequality (\ref{coeff}) guarantees  that 
$\sum_{\zeta \in  \X \setminus \Upsilon(\xi)} |A_{\xi,\zeta}|
\le
Cq^{d-2m} \sum_{|\xi-\zeta|\ge Kh|\log h|} \exp\left(-\nu\frac{\dist(\xi,\zeta)}{h} \right)
$.
 By observing that for $\zeta\in \X\setminus \Upsilon(\xi)$, we have $q^d \le C \vol\bigl(B(\zeta,q) \setminus B(\xi,Kh|\log h|)\bigr)$ with a constant $C$ that depends only on the spatial dimension $d$. 
 (Note that for most $\zeta$, the above set is simply $B(\zeta,q)$, while for those $\zeta$ which are near the boundary of $B(\xi,Kh|\log h|)$ the set contains a half-ball), we 
 can control the above sum by an integral, namely
\begin{eqnarray}
\sum_{\zeta \in  \X \setminus \Upsilon(\xi)} |A_{\xi,\zeta}|&\le& 
Cq^{d-2m} \sum_{|\xi-\zeta|\ge Kh|\log h|} \exp\left(-\nu\frac{|\xi-\zeta|}{h} \right)\nonumber\\
&\le& Cq^{-2m}
\sum_{|\xi-\zeta|\ge Kh|\log h|} \int_{y\in B(\zeta,q) \setminus B(\xi,Kh|\log h|)}
 \exp\left(-\nu\frac{|\xi-\zeta|}{h} \right)\dif y\nonumber\\
&\le& Cq^{-2m}
\int_{y\in \reals^d \setminus B(\xi,Kh|\log h|)}
 \exp\left(-\nu\frac{|\xi-\zeta|}{h} \right)\dif y
 \label{sum_of_truncated_coeffs}
  \end{eqnarray}
In the final inequality, we have used the fact that the sets $B(\zeta,q)\setminus B(\xi,Kh|\log h|)$ are disjoint
and that for $y\in  B(\zeta,q)$, $\dist(\xi,y)\le \dist(\xi,\zeta) + q\le \dist(\xi,\zeta)+h$, which implies
$- \dist(\xi,\zeta) \le - \dist(\xi,y) +h$ (leading to a small increase in the constant; a factor of $e^{\nu}$).

Applying a polar change of variables in the final integral gives the inequality
\begin{equation*}
\sum_{\zeta \in  \X \setminus \Upsilon(\xi)} |A_{\xi,\zeta}|
\le Cq^{-2m}
\int_{Kh|\log h|}^{\infty}
 \exp\left(-\nu\frac{r}{h} \right)r^{d-1}\dif r.
\end{equation*}
We simplify this estimate by splitting $\nu = \nu/2 +\nu/2$ and writing
\begin{eqnarray*}
\sum_{\zeta \in  \X \setminus \Upsilon(\xi)} |A_{\xi,\zeta}|
&\le&
C h^d q^{-2m}
\left(\int_{K|\log h|}^{\infty} r^{d-1}
 \exp\left(- K |\log h| \frac{\nu}{2} \right) 
 \exp\left(- r \frac{\nu}{2} \right)\dif r\right)\\
 &\le& 
 C h^d q^{-2m}
 h^{K\nu/2}
 = C \rho^{2m} h^{K\nu/2 +d-2m}.
\end{eqnarray*}
The lemma follows.
\end{proof}

\subsubsection{ Bounds for truncated functions:  Mat{\' e}rn functions} 

Let $\| \cdot \|_Z$  be a norm on $S( \X)$ for which a universal constant $\Gamma $
exists
so that $\sup_{z\in\Omega} \|k_m(\cdot-z)\|_Z \le \Gamma$.
Since $\|k_m(\cdot-z)\|_Z$ is finite and bounded independent of $z$,
we have
\begin{equation}\label{general_norm}
\|\chi_{\xi} - \wch_{\xi}\|_Z \le\Gamma \sum_{\zeta \in  \X \setminus \Upsilon(\xi)} |A_{\xi,\zeta}|  \le C \Gamma 
 \rho^{2m}h^{K\nu/2-2m+d}.
 \end{equation}
  In particular, we have the following:
  \begin{lemma} \label{Mat}
Let $m>d/2$  and consider the  Mat{\' e}rn radial basis function $k_m = \kappa_m$ described in (\ref{Matern}).
  For $1\le p< \infty$ and 
  $\sigma < 2m-d + \frac{d}{p}$   
  we have 
  $$
\|\chi_{\xi} - \wch_{\xi} \|_{W_p^\sigma(\reals^d)} 
 \le 
 \sum_{\zeta \in  X \setminus \Upsilon(\xi)  } |A_{\xi,\zeta}|  \left\| \kappa_m(\cdot, \zeta) \right\|_{W_p^\sigma(\reals^d)} 
\le 
C\rho^{2m}h^{K\nu/2 +d -2m}
 $$
 with $C=C(m,d)$.

 For $p=\infty$, the above result holds for the H{\"o}lder space $W_{\infty}^\sigma(\reals^d)$ replaced with $C^{\sigma}(\reals^d)$
\end{lemma}
\begin{proof}
We have from \cite{HNRW1}[Lemma A.1] that 
$\kappa_m \in W_p^{\tau}(\reals^d)$ for $1\le p< \infty$ and $\tau < 2m - d+d/p$, while for
$p=\infty$, $\kappa_m \in C^{\tau}(\reals^d)$ with $\tau < 2m - d$. 
In either case, the smoothness norm is translation invariant, so it follows that
   $$\|\kappa_m(\cdot - z)\|_{ W_p^{\tau}(\reals^d)}\le C_{\tau,p}\qquad\text{and}\qquad
    \|\kappa_m(\cdot - z)\|_{ C^{\tau}(\reals^d)}\le C_{\tau,\infty}$$
hold. The result follows from (\ref{general_norm}).
 \end{proof}

\subsubsection{Bounds for truncated functions: Surface splines}
\label{CPD_truncated}
When $k_m=\phi_m$ (i.e., a surface spline, and therefore conditionally positive definite),
 the  argument of the previous section is  a little more complicated. 
  Given a Lagrange function $\chi_{\xi} = \sum_{\zeta\in  X} A_{\zeta,\xi} k_m(\cdot,\zeta) + p$,
 simply truncating coefficients does not yield a function in $S(\Upsilon(\xi))$. 
 That is, $( A_{\zeta,\xi} )_{\zeta\in \Upsilon(\xi)}$
 does not necessarily satisfy the side condition $\sum_{\zeta\in \Upsilon(\xi)} A_{\zeta,\xi} p(\zeta) = 0$ for all $p\in \Pi_{m-1}$.

The result for {\em restricted surface splines} on even dimensional spheres 
($\mathbb{S}^{2n}$)
has been developed in \cite[Proposition 6.1]{FHNWW}.
We now present a similar estimate for surface splines on $\reals^d$
where the truncated Lagrange function is corrected by perturbing its coefficients slightly. 
This is done by using the orthogonal projector having range $ \perp (\Pi_{m-1}\left|_{\Upsilon(\xi)}\right.)$.
Keeping this perturbation small is essential to our later results, so we must estimate it. 
We use the following result about  Gram matrices for polynomials sampled on finite point sets.

\paragraph{Gram matrices for polynomials restricted to point sets}
Let  $N = \dim \Pi_{m-1}$ and consider $X\subset \reals^d$ a finite point set.
For a basis $\{p_1,\dots, p_N\}$ of $\Pi_{m-1}$,  denote by 
 $\Phi_{X}$ the (Vandermonde-type) matrix
  with $N$ columns and $\# X$ rows
  whose $j^{th}$ column is $p_j$ restricted to $X$. In other words,
  \begin{equation}\label{Phi}
  \Phi_{X}\in M_{ (\# X)\times N}(\reals)\qquad \text{with} \qquad (\Phi_{X})_{\xi,j} = p_j(\xi).
  \end{equation}

  \begin{lemma}\label{Gram_assumption}
  For every $m\in \nats$, and any radius $r>0$, point $x\in\reals^d$
  and  point set $X\subset B(x,r)$ with fill distance $h\le h_0 r$, where $h_0=h_0(m)$,
  the inverse of the Gram matrix $\G_X = \Phi_X^T \Phi_X\in M_{N\times N}(\reals)$ has norm bounded by
 $$\|\G_X^{-1}\|_{2\to 2} \le   C r^{-2(m-1)} $$
 with $C= C(m,d)$.
 \end{lemma}
 \begin{proof}
 From \cite[Theorem 3.8 and Corollary 3.11]{Wend}, we have that if $X\subset B(x,r)$ has fill distance $h\le h_0 r$, then
 $X$ is a {\em norming} set for $B(x,r)$ with norming constant $2$. (Here $h_0 = c_{m-1}$, from \cite[Corollary 3.11]{Wend}.)  This means that for every $p\in \Pi_{m-1}$,
 $ \|p \|_{L_{\infty}(B(x,r))}\le 2  \|p\left|_{X}\right.\|_{\ell_\infty(X)}$.
 
 The norm of the Gram matrix can be controlled  by %a simple norming set argument, since
$$\| \G_X^{-1}\|_{2\to 2} = (\min_{\|\bfa\|=1}\langle \G_X \bfa,\bfa\rangle )^{-1} \quad \text{ and } \quad
\langle \G_X \bfa, \bfa\rangle =  \|\Phi_X \bfa \|_{\ell_2(X)}^2 = \| R_{X} V \bfa\|_{\ell_2(X)}^2$$
where $V \bfa := \sum_{j=1}^N a_j p_j$ and 
$R_{X}$ is the restriction operator $R_{X} V \bfa = \sum_{j=1}^N a_j p_j\left|_{X}\right.$.
For $h$ sufficiently small, the norming set property ensures that
\begin{equation*}
 \|p \|_{L_{\infty}(B(x,r))}
\le
2  \|R_{X} p\|_{\ell_\infty(X)}
\le 2   
\| R_{X} p\|_{\ell_2(X)}.
\end{equation*}
On the other hand, we have the following growth properties of polynomials $\Pi_{m-1}$: there exists a constant $C_m>0$ so that
for every $0<r<1$,
$\|p\|_{L_{\infty}(B(x,1))} \le C_m r^{-(m-1)} \|p\|_{L_{\infty}(B(x,r))}$. 
Returning to the basis $(p_1,\dots ,p_N)$, we have
$$ 
\|\bfa \|_{\ell_2(N)}
\le
C_{m,d}\|\sum_{j=1}^N a_j 
p_j \|_{L_{\infty}(B(x,1))}
\le
C_{m,d}\left(\frac{1}{r}\right)^{m-1}
\|\sum_{j=1}^N a_j p_j \|_{L_{\infty}(B(x,r))}.
$$
This gives 
 $\|\bfa\|_{\ell_2(N)} \le C  r^{-(m-1)} 
 \|  \sum_{j=1}^N a_j p_j\left|_{X}\right. \|_{\ell_2(X)}$, 
 and the result follows.
\end{proof}
A bound similar to this for $\S^{d-1}$ using spaces of spherical harmonics in place of $\Pi_{m-1}$ has been demonstrated in \cite[Lemma 6.4]{FHNWW},
while \cite{HNRW1} gives general conditions for the auxiliary space of a CPD kernel.

\paragraph{Modifying coefficients}
 We use the matrix $\Phi_{\Upsilon(\xi)}$ 
to construct  
$P =\Phi_{\Upsilon(\xi)}(\Phi_{\Upsilon(\xi)}^T \Phi_{\Upsilon(\xi)})^{-1}\Phi_{\Upsilon(\xi)}^T$,
the orthogonal projector
 which has range
$\Pi\left|_{\Upsilon(\xi)}\right.$ and kernel $\perp\!(\Pi\left|_{\Upsilon(\xi)}\right.)$. 
For a fixed $\xi$, denote the truncated coefficients $ (A_{\zeta,\xi})_{\zeta\in\Upsilon(\xi)} \in \reals^{\Upsilon(\xi)}$ by $\bfA_{\xi}$. 
In order to satisfy the side conditions, we generate the modified coefficients 
$\wbfA_{\xi} = (\wA_{\zeta,\xi}) \in \reals^{\Upsilon(\xi)}$ via
$$\wbfA_{\xi}= \bfA_{\xi} - P\bfA_{\xi}.$$ 
In other words, $\wbfA_{\xi}$ is the orthogonal projection of $\bfA_{\xi}$ onto $\perp (\Pi\left|_{\Upsilon(\xi)}\right.)$.
Define the `truncated' Lagrange function as
$$\wch_{\xi} := \sum_{\zeta\in \Upsilon(\xi)} \wA_{\zeta,\xi} \phi_m(\cdot-\zeta) +p.$$
  
  %%%%%%%%%%%%%%%%%%%%%%%%%
  %%%%
  %%%%
  %%%%
  %%%%
  %%%%%%%%%%%%%%%%%%%%%%%%%
  \begin{lemma}\label{ell2upsilon}
  Let $m>d/2$  and consider the  surface spline radial basis function $k_m = \phi_m$ described in (\ref{surface_spline}).
  For sufficiently small $h$ we have
  \begin{equation}\label{truncation_error}
  \|\bfA - \wbfA\|_{\ell_2(\Upsilon(\xi))}\le C\rho^{2m}h^{K\nu/2+1 -3m+d}  |\log h|^{1-m}
  \end{equation}  
  with $C=C(m,d)$.
  \end{lemma}
  \begin{proof}
  
  We estimate the $\ell_2$ norm of the difference  of the coefficients as
$$
\|\bfA_{\xi} - \wbfA_{\xi}\|_{\ell_2(\Upsilon(\xi))} 
=
\|P\bfA_{\xi}\|_{\ell_2(\Upsilon(\xi))}
= \langle \Phi_{\Upsilon(\xi)}^T \bfA_{\xi}, \G_{\Upsilon(\xi)}^{-1}\Phi_{\Upsilon(\xi)}^T\bfA_{\xi}\rangle^{1/2} 
\le 
\| \G_{\Upsilon(\xi)}^{-1}\|_{2\to 2}^{1/2} \|\Phi_{\Upsilon(\xi)}^T \bfA_{\xi}\|_{\ell_2(N)} .
$$
Since $\sum_{\zeta\in \X} A_{\zeta,\xi} p(\xi) = 0$  for all $p\in\Pi $, we have
$
\Phi_{\Upsilon(\xi)}^T \bfA_{\xi} 
= 
-\bigl(\sum_{\zeta\in \X\setminus \Upsilon(\xi)} A_{\zeta,\xi} p_j(\zeta)\bigr)_{j=1}^N 
$. 

Applying the estimate (\ref{decreasing})
the
 $\ell_2(N)$ norm of $\Phi^T \bfA$ is controlled by
\begin{equation*}
 \|\Phi_{\Upsilon(\xi)}^T \bfA_{\xi}\|_{\ell_2(N)}\le  
 \|\Phi_{\Upsilon(\xi)}^T \bfA_{\xi}\|_{\ell_1 (N)}
 \le
\sum_{j=1}^N\left| \sum_{\zeta\in \X\setminus \Upsilon(\xi)} A_{\zeta,\xi} p_j(\zeta)\right|
\le
\sum_{\zeta\in \X\setminus \Upsilon(\xi)} \left| A_{\zeta,\xi}\right| \sum_{j=1}^N  |p_j(\zeta)|.
 \end{equation*}
  In the first estimate we use the inequality $\sum |c_j|^2 \le(\sum |c_j|)^2$.
 Applying H{\"o}lder's inequality and (\ref{coeff}) to the right hand side gives
 $$
  \|\Phi_{\Upsilon(\xi)}^T \bfA_{\xi}\|_{\ell_2(N)}\le  
  N
 \sum_{\zeta\in \X\setminus \Upsilon(\xi)}(\max_{j=1\dots N} |p_j(\zeta)|) \left| A_{\zeta,\xi}\right|
 \le 
C q^{d-2m} \sum_{\zeta\in \X\setminus \Upsilon(\xi)}(\max_{j=1\dots N} |p_j(\zeta)|) \mathrm{exp}\left(-{\nu} \frac{\dist(\xi,\zeta)}{h}\right).
 $$
where we have absorbed $N$ (recall that $N= \dim \Pi_{m-1}$ depends on $m$ and $d$) into the constant $C$.

 We now recall the argument in (\ref{sum_of_truncated_coeffs}) which allows us to estimate the above sum by an integral:
 \begin{eqnarray}
 \|\Phi_{\Upsilon(\xi)}^T \bfA_{\xi}\|_{\ell_2(N)}
 &\le&Cq^{-2m}\int_{Kh\log h}^{\infty} \max_{j=1\dots N}\left(\|p_j\|_{L_{\infty}\bigl(B(\xi,z)\bigr)}\right) 
 e^{-\nu z/h} \dif z\nonumber\\
&\le&
C q^{-2m}\int_{Kh\log h}^{\infty} 
\max(1,z^{m-1}) 
e^{-\nu z/h} \dif z\nonumber\\
&\le &
C \rho^{2m} h^{K\nu/2 + d-2m}.\label{PhiTA_int}
\end{eqnarray}
In (\ref{PhiTA_int}) we have used a change to polar coordinates, as in Lemma \ref{trunc_coeffs}.

Estimate (\ref{truncation_error}) follows by combining Lemma \ref{Gram_assumption} (using $r = Kh|\log h|$) with (\ref{PhiTA_int}).
\end{proof}

As in the positive definite case, we are able to control the truncation error measured in suitable smoothness norms 
- the only requirement is that the kernel is bounded. 
In the conditionally positive definite case, 
the kernel may be unbounded, so we measure the norm over the bounded region $\Omega$.
Specifically, the surface spline  $\phi_m\in W_{p,\mathrm{loc}}^{\sigma}(\reals^d)$ for all $\sigma<2m-d+\frac{d}{p}$ 
(as well as $C_{\mathrm{loc}}^{\sigma}(\reals^d)$ for
    $\sigma< 2m-d$).
There is $\Gamma<\infty$ (depending on  $\sigma$, $p$, $m$ and $\Omega$) so that for $\zeta\in \Om$, 
$\| \phi_m(\cdot-\zeta)\|_{W_{p}^\sigma(\Omega)}\le \Gamma$.

\begin{lemma}\label{trunc_sob_error}
for $1\le p< \infty$ and 
  $\sigma < 2m-d + \frac{d}{p}$   
    $$\| \wch_{\xi} - \chi_{\xi} \|_{W_p^{\sigma}({\Omega})} \le C\rho^{2m+d/2}h^{ K\nu/2 +1-3m+d }|\log h|^{d/2+1-m}.$$
    with $C = C(\sigma,m,p,\Omega)$.
    
    A similar result holds for $p=\infty$, replacing $W_p^{\sigma}(\Omega)$ by $C^{\sigma}(\overline{\Omega})$ for
    $\sigma< 2m-d$.
    \end{lemma}
    \begin{proof}
    The Sobolev estimate holds by considering
$$\|\wch_{\xi} - \chi_{\xi}\|_{W_{p}^\sigma(\Omega)} \le
 \sum_{\zeta \in \Upsilon(\xi)} |A_{\xi,\zeta}- \widetilde{A}_{\xi,\zeta}| \| \phi_m(\cdot-\zeta)\|_{W_{p}^\sigma(\Omega)}
+\sum_{\zeta \notin \Upsilon(\xi)} |A_{\xi,\zeta}| \| \phi_m(\cdot-\zeta)\|_{W_{p}^\sigma(\Omega)}.$$

The first term can be bounded by introducing the constant $\Gamma := \max_{\zeta\in \Om} \| \phi_m(\cdot-\zeta)\|_{W_{p}^\sigma(\Omega)} <\infty$,
which gives 
$ \sum_{\zeta \in \Upsilon(\xi)} |A_{\xi,\zeta}- \widetilde{A}_{\xi,\zeta}| \|\phi_m(\cdot-\zeta)\|_{W_{p}^\sigma(\Omega)}
  \le \Gamma \|\bfA_{\xi} - \wbfA_{\xi}\|_{\ell_1(\Upsilon(\xi))} $.
Employing (\ref{truncation_error}) yields
\begin{eqnarray*}
 \sum_{\zeta \in \Upsilon(\xi)} 
 |A_{\xi,\zeta}- \widetilde{A}_{\xi,\zeta}| \|\phi_m(\cdot-\zeta)\|_{W_{p}^\sigma(\Omega)}
  &\le&C\Gamma\rho^{2m}
  h^{K\nu/2+1 -3m+d}  |\log h|^{1-m}(\#\Upsilon(\xi))^{d/2} 
  \\
&\le&C\Gamma \rho^{2m+d/2} h^{\frac{K\nu}{2} +1-3m+d}  |\log h|^{1-m+d/2}.  
\end{eqnarray*}
For the second inequality we have used the estimate
$\#\Upsilon(\xi)\le C \rho^{d} |\log h|^{d}$.

The second term is bounded by
$
\sum_{\zeta \notin \Upsilon(\xi)} |A_{\xi,\zeta}| \|\phi_m(\cdot-\zeta)\|_{W_p^\sigma(\Omega)}
\le
\Gamma \sum_{\zeta \notin \Upsilon(\xi)} |A_{\xi,\zeta}| $
which can be further treated with Lemma \ref{trunc_coeffs}
 to obtain 
 $
\sum_{\zeta \notin \Upsilon(\xi)} |A_{\xi,\zeta}| \|\phi_m(\cdot-\zeta)\|_{W_p^\sigma(\Omega)}
\le
 C\Gamma \rho^{2m} h^{\frac{K\nu}{2}-2m+d}.
$
 \end{proof}

\subsection{Local Lagrange Functions}\label{SS:LLF}
In this section we consider a locally constructed function $b_{\xi}$.
Our main goal is to show that for $\Xi\subset \Omega$, there exist functions $b_{\xi}$ defined on
$\reals^d$, so that $\|\sum_{\xi\in\Xi}a_\xi b_{\xi} \|_{W_p^\sigma(\Omega)}\le Ch^{\frac{d}{p} - \sigma} \|\bfa\|_{\ell_p(\Xi)}$.

At this point, a standard argument  bounds the error between $b_{\xi}$ and $\wch_{\xi}$ (this argument is 
essentially the same one used on the sphere in \cite{FHNWW}). 
This works by measuring the size of 
$b_{\xi} - \wch_{\xi} \in S(\Upsilon(\xi))$.

%%%%
%
%
%
%%%%
\subsubsection{Bounds for local Lagrange functions:  Mat{\'e}rn  functions}\label{SSS:Local_Matern}
For the positive definite case, the argument is fairly elementary.
For $\zeta\in \Upsilon(\xi)$, let $y_{\zeta} := b_{\xi}(\zeta) - \wch_{\xi} (\zeta)$.
Observe that $b_{\xi}- \wch_{\xi} = \sum_{\zeta\in \Upsilon(\xi)} a_{\zeta} k_m(\cdot-\zeta) \in S(\Upsilon(\xi))$, 
where $\bfa = (a_{\zeta})$ and $\bfy = (y_{\zeta})$ are related by
 $\K_{\Upsilon(\xi)} \bfa = \bfy$. 
 The matrix $(\K_{\Upsilon(\xi)})^{-1}$ has entries $(A_{\zeta,\eta} )_{\zeta,\eta\in\Upsilon(\xi)}$.
 
 For a kernel of order $m$,  the entries of the matrix $\mathcal{A} = (A_{\zeta,\eta} )_{\zeta,\eta\in\Upsilon(\xi)}$
 can be estimated by (\ref{basiccoeffs}): 
 $|A_{\zeta,\eta}| \le  Cq^{d-2m}$. 
 It follows that $(\K_{\Upsilon(\xi)})^{-1} $ has
 $\ell_1$ matrix norm 
 $$\left\|\left(\K_{\Upsilon(\xi)}\right)^{-1}\right \|_{1\to 1} \le C (\# \Upsilon(\xi)) q^{d-2m}
 \le C  \rho^{2m} |\log h|^d h^{d-2m}.$$
 (Here we have used the estimate $\# \Upsilon(\xi)\le C \rho^d |\log h|^d$.) 
Consequently
$\|\bfy\|_1 \le  
(\# \Upsilon(\xi)) 
\|\bfy\|_{\infty} $. 

Because $y_{\zeta} = \chi_{\xi}(\zeta) - \wch_{\xi} (\zeta)$ for $\zeta\in  \Upsilon(\xi)$ and $\|\chi_{\xi} - \wch_{\xi}\|_{\infty}\le C \|\chi_{\xi} - \wch_{\xi}\|_{W_2^m(\reals^d)}$ we have
$$
\sum_{\zeta\in \Upsilon(\xi)} |a_\zeta| 
\le 
\left\|\left(\K_{\Upsilon(\xi)}\right)^{-1}\right \|_{1\to 1}
\|\bfy\|_1
\le
 C  \rho^{2m+d}h^{d-2m}   |\log h|^{2d} \|\chi_{\xi} - \wch_{\xi}\|_{W_2^m(\reals^d)}.
$$

For a generic norm $\|\cdot \|_Z$ for which $\max_{z \in \Om} \|k_m(\cdot - z)\|_Z \le \Gamma$ we have
$ \|b_{\xi} - \wch_{\xi}\|_Z\le \Gamma \sum_{\zeta} |a_{\zeta}|$.
We now have the counterpart to Lemma \ref{Mat}, which shows that (\ref{compatible}) holds for Mat{\'e}rn kernels.
\begin{lemma}\label{Matern_compare}
For $k_m = \kappa_m$, and  for $1\le p\le \infty$ and $\sigma < 2m -d +d/p$ we have
\begin{equation}%\label{local_lagrange}
\left\|b_{\xi} -\chi_{\xi}\right \|_{W_p^{\sigma}(\reals^d)}
\le
C \rho^{4m+d}h^{K\nu/2 +2d -4m}|\log h|^{2d}
\end{equation}
with $C=C(m,d)$.
\end{lemma}

Setting $|\log h|^{2d} \le C h^{-1}$ (either by finding a sufficiently small $h^*$ so that this holds for $h<h^*$,
 or by increasing the constant, or both), and by employing a simple interpolation inequality, we have
\begin{equation}\label{Sob_loc_full_error}
\left\|b_{\xi} -\chi_{\xi}\right \|_{W_p^{\sigma}(\reals^d)}
\le C  
\rho^{4m+d}h^J, \ J=K\nu/2 +2d -4m-1.
\end{equation}

%%%%%
%
%
%
%%%%%
\subsubsection{Bounds for local Lagrange functions: Surface splines}\label{SSS:Local_surface}
As in the previous section,  we are guided by the estimates for local Lagrange functions on the sphere
\cite[Proposition 5.2]{FHNWW}.

In this case we have 
$\wch_{\xi} - b_{\xi} 
= 
\sum_{\zeta\in\Upsilon(\xi)} a_{\zeta} \phi_m(\cdot -\zeta) 
+
\sum_{j=1}^N c_j p_j\in S(\Upsilon(\xi))
$.
The vectors  $\bfa = (a_\zeta)_{\zeta\in\Upsilon(\xi)}$ and $\bfc = (c_j)_{j=1\dots N }$ are related to 
$\bfy = (y_{\zeta})_{\zeta\in \Upsilon(\xi)} = (\wch_{\zeta} - b_{\zeta})_{\zeta\in \Upsilon(\xi)}$ by
$$
\left(\begin{matrix} 
\K_{\Upsilon(\xi)} &\Phi\\ \Phi^T &0_{N\times N} 
\end{matrix}\right)
\begin{pmatrix}\bfa \\ \bfc\end{pmatrix} 
=\begin{pmatrix} \bfy\\ 0_{N\times 1}\end{pmatrix}
$$
where $\K_{\Upsilon(\xi)}$ is the collocation matrix and  $\Phi$ is the Vandermonde matrix introduced in (\ref{Phi}).
The norms of $\bfa$ and $\bfc$ can be controlled by  $\|\bfy\|_{\ell_2(\Upsilon(\xi))}$. 
This is demonstrated in \cite[Proposition 5.2]{FHNWW}, which shows that
$
\|\bfa\|_{\ell_2(\Upsilon(\xi))} \le \vartheta^{-1} \|\bfy\|_{\ell_2(\Upsilon(\xi))}
$
where $\vartheta$ is the minimal positive eigenvalue of $P^{\perp}\K_{\Upsilon(\xi)}P^{\perp}$.
Recall that  $P^{\perp}= \mathrm{Id} - P$ and 
$P = \Phi (\Phi^T \Phi)^{-1} \Phi^T$ is the projector introduced in 
Section \ref{CPD_truncated}.

We make the following observation, which is \cite[Proposition 5.2]{FHNWW}:
\begin{eqnarray}
\|\bfa\|_{\ell_2(\Upsilon(\xi))} &\le& \vartheta^{-1} \|\bfy\|_{\ell_2(\Upsilon(\xi))} \qquad \text{and}\label{a_bound}
\\ 
\|\bfc \|_{\ell_2(N)} &\le& 
2
\max_{\eta,\zeta\in \Upsilon(\xi)}|\phi_m(\eta-\zeta)|
 \|\G_{\Upsilon(\xi)}^{-1}\|^{1/2}
\vartheta^{-1}
(\# \Upsilon(\xi))
\|\bfy\|_{\ell_2(\Upsilon(\xi))}\label{c_bound}
\end{eqnarray}
It is possible to estimate the size of $\vartheta$ by considering the matrix of kernel coefficients for the Lagrange functions
$b_{\eta,\Upsilon(\xi)} 
= 
\sum_{\zeta\in\Upsilon(\xi)}A_{\zeta,\eta} k_m(\cdot,\zeta) +\sum_{j=1}^N B_{j,\eta}\phi_j$.
\begin{lemma}\label{CPD_Matrix_norm}
For $\vartheta$, the minimal positive eigenvalue of $P^{\perp}\K_{\Upsilon(\xi)}P^{\perp}$, we have
$\vartheta^{-1} = \|\A\|_{2\to 2}$, 
where $\A = (A_{\zeta,\eta})_{\zeta,\eta\in\Upsilon(\xi)}$ is the matrix of kernel coefficients for
the Lagrange functions in $ S(\Upsilon(\xi))$. 
\end{lemma}
 \begin{proof}
Writing
$\B = (B_{j,\eta})_{\substack{j=1\dots N\\ \eta\in \Upsilon(\xi)}}$
it follows that
$\K_{\Upsilon(\xi)} \A + \Phi \B = \mathrm{Id}$. 
From this we have $P^{\perp} = P^{\perp} \K_{\Upsilon(\xi)} \A$
and $\ker \A \subset \ker P^{\perp}$. 
On the other hand, each column of $\A$ satisfies the
side condition 
$\sum_{\eta\in\Upsilon(\xi)}A_{\zeta,\eta} p(\eta) = 0$ 
for all $p\in \Pi$, so $\mathrm{ran } \A \subset \mathrm{ran } P^{\perp}$. 
From this it follows that 
$\ker \A = \ker P^{\perp}$ and 
$\mathrm{ran } \A = \mathrm{ran } P^{\perp}$.

Because $P^\perp \A = \A$ 
we have 
$
P^{\perp} 
= 
P^{\perp} \K_{\Upsilon(\xi)} \A 
= 
P^{\perp} \K_{\Upsilon(\xi)} P^{\perp} \A
$,
and the nonzero spectrum of $\A$ is the reciprocal of the nonzero spectrum of 
$P^{\perp} \K_{\Upsilon(\xi)} P^{\perp} $.
In other words, 
$\vartheta^{-1} = \max_{\lambda\in \sigma(\A)} |\lambda|$.
\end{proof}

Applying Gershgorin's theorem to $\A$, whose entries are $A_{\zeta, \eta} = \langle b_{\zeta,\Upsilon(\xi)}, b_{\eta,\Upsilon(\xi)}\rangle$ 
and therefore satisfy $|A_{\zeta,\eta}| \le C q^{d-2m}$, we have 
$\vartheta^{-1}  \le C\bigl(1+\#\bigl(\Upsilon(\xi)\bigr)\bigr) q^{d-2m}$. By (\ref{a_bound}) we have
\begin{equation}
\label{Gersh}
\|\bfa\|_{\ell_2(\Upsilon(\xi))} \le C \rho^{2m} h^{d-2m} |\log h|^d \|\bfy\|_{\ell_2(\Upsilon(\xi))}.
\end{equation}

Using Lemma \ref{Gram_assumption}, we have that 
$\|\G_{\Upsilon(\xi)}^{-1}\|^{1/2}\le C_{m,d}(K h |\log h|)^{-(m-1)} $,
while $ \vartheta^{-1} \le C_{m,d} \rho^{2m} h^{d-2m} |\log h|^d$
and
$(\# \Upsilon(\xi))\le C_{m,d} \rho^{2m} |\log h|^d$.
Applying  (\ref{c_bound}) gives
\begin{eqnarray}\label{c_estimate}
\|\bfc\|_{\ell_2(N)} &\le&
C_{m,d}
(2K h |\log h|)^{2m-d}
\left((K h |\log h|)^{-(m-1)}   \right)
\left( \rho^{2m} h^{d-2m} |\log h|^d\right)
(\rho^d |\log h|^d)
\nonumber \\
&\le& C_{m,d}
\rho^{2m+d}
h^{-(m-1)}
|\log h|^{m+1+d}
\|\bfy\|_{\ell_2(\Upsilon(\xi))}
\end{eqnarray}

We are now in a position to prove that %a basic result about local basis functions.
 (\ref{compatible}) 
holds for surface splines.
\begin{lemma}\label{ss_comparison}
Let $k_m = \phi_m$, the surface spline RBF on $\reals^d$ and let $J>0$. For $\Xi\subset \Omega$,
form the local Lagrange functions $b_{\xi} \in \Upsilon(\xi)$, with $\Upsilon(\xi) = \X\cap B(\xi, K h|\log h|)$, where
 $J =K\frac{\nu}{2} -5m +d+1$.
Then for  $1\le p\le \infty$ and  
  $\sigma < 2m-d + \frac{d}{p}$,
and for sufficiently small $h$,      
$$\|b_{\xi} - \chi_{\xi}\|_{W_p^{\sigma}(\Omega)}\le C\rho^{4m+2d} h^J$$
with $C= C(\sigma,m,p,\Omega)$.
\end{lemma}
 \begin{proof}
We use the triangle inequality 
$
\|b_{\xi} - \chi_{\xi}\|_{W_p^{\sigma}(\Omega)}
\le  
\|b_{\xi} - \wch_{\xi}\|_{W_p^{\sigma}(\Omega)} +
\| \wch_{\xi}-  \chi_{\xi}\|_{W_p^{\sigma}(\Omega)}$, 
noting that the second term has been estimated in Lemma \ref{ell2upsilon}, and that the first can be controlled as
$$
\|b_{\xi} - \wch_{\xi}\|_{W_p^{\sigma}(\Omega)}
 \le
\|\bfa\|_{\ell_1(\Upsilon(\xi))} \max_{z\in \Upsilon(\xi)}\|\phi_m(\cdot - z)\|_{W_p^{\sigma}(\Omega)}
+
\|\bfc\|_{\ell_1(N)} \max_{1\le j\le N}\|\phi_j(\cdot)\|_{W_p^{\sigma}(\Omega)}.
$$
From (\ref{Gersh}) we have 
$
\|\a\|_{\ell_1(\Upsilon(\xi))}
\le 
\sqrt{\#\Upsilon(\xi)} \|\a\|_{\ell_2(\Upsilon(\xi))}
$ 
and $\#\Upsilon(\xi)\le C_{m,d} \rho^d |\log h|^d$,
so
\begin{eqnarray*}
\|\bfa\|_{\ell_1(\Upsilon(\xi))} &\le& C \rho^{2m+d/2} h^{d-2m} |\log h|^{3d/2} \|\bfy\|_{\ell_\infty(\Upsilon(\xi))}\\
&\le&
C\rho^{4m+d} h^{ K\nu/2 -5m+2d+1 }|\log h|^{2d-(m-1)},
\end{eqnarray*}
where we have employed the result of Lemma \ref{trunc_sob_error} and the embedding $W_2^m\subset L_{\infty}$ to estimate
$\|\bfy\|_{\ell_\infty(\Upsilon(\xi))} \le \|b_{\xi} - \wch_{\xi}\|_{L_\infty(B(\xi,Kh|\log h|))}\le C_{d,m} \rho^{2m+d/2} h^{K\nu/2 +1-3m+d}|\log h|^{1-m+d/2}$.

Similarly, from (\ref{c_estimate}), we have
\begin{eqnarray*}
\|\bfc\|_{\ell_1(N)} 
&\le& 
C
\rho^{2m+d+d/2}
h^{-(m-1)}
|\log h|^{m+1+d+d/2}
\|\bfy\|_{\ell_\infty(\Upsilon(\xi))}\\
&\le&
C
\rho^{4m+2d}
h^{ K\nu/2 -4m+2+d}
|\log h|^{2+2d}.
\end{eqnarray*}
Because  $\max_{z\in \Upsilon(\xi)}\|\phi_m(\cdot - z)\|_{W_p^{\sigma}(\Omega)}$ and
$ \max_{1\le j\le N}\|\phi_j(\cdot)\|_{W_p^{\sigma}(\Omega)}$ are bounded by a constant $\Gamma$ which
depends only on $\Omega$, $m$, $p$ and $\sigma$, we have
$$\|b_{\xi} - \chi_{\xi}\|_{W_p^{\sigma}(\Omega)}
\le  
\Gamma C
\rho^{4m+2d}
h^{ K\nu/2 -5m+1+d}
|\log h|^{2+2d}.
$$
The lemma follows for $h$ sufficiently small that $|\log h|^{2d+2}<h^{-1}$.
 \end{proof}

%%%%%%%%%%%%%%%%%%%%%%%%%%%%
%
%%%%%%%%%%%%%%%%%%%%%%%%%%%%%
\subsection{Bernstein type estimate for local Lagrange functions}\label{local_bernstein}

In this section we discuss the  local Lagrange ($b_{\xi}$) functions generated by $k_m$ and the centers $ \X$.
We develop partial Bernstein inequalities similar to (\ref{LP-Synthesis}), 
where for functions $\sum_{\xi\in  \Xi} a_{\xi} b_{\xi}\in \V_{\Xi}$, smoothness norms
$\|s\|_{W_p^{\sigma}}$ are controlled by an $\ell_p$ norm on the coefficients: $\|\bfa\|_{\ell_p( \Xi)}$.

\begin{theorem}\label{main_local_bernstein}
Consider the family of local Lagrange functions generated with $K>\frac{10m-2}{\nu}$. 
For  $0\le \sigma \le m-(d/2-d/p)_+$ 
when $1\le p< \infty$ (or $\sigma \in \nats$ and  $0\le \sigma<m - d/2$ when $p=\infty$), we have
\begin{equation}\label{first_inverse_estimate}
\bigg\|\sum_{\xi\in  \Xi} a_{\xi} b_{\xi}\bigg\|_{W_p^{\sigma}(\Omega)} \le C \rho^{4m+3d} h^{d/p - \sigma} 
\|\bfa\|_{\ell_p( \Xi)}
\end{equation}
where $C=C(\sigma, p, m, \Omega)$.
\end{theorem}
\begin{proof}
We start with the basic splitting 
$$
s := \sum_{\xi\in  \Xi} a_{\xi} b_{\xi} 
=\bigg(  \sum_{\xi\in  \Xi} a_{\xi} \chi_{\xi}\bigg)+\bigg(\sum_{\xi\in  \Xi} a_{\xi}( b_{\xi}-\chi_{\xi} )\bigg) 
=: G+B .
$$ 
Applying the Sobolev norm gives
$\|s\|_{W_p^{\sigma}(\Omega)} \le \|G\|_{W_p^{\sigma}(\Omega)} +\|B\|_{W_p^{\sigma}(\Omega)}$.
From (\ref{LP-Synthesis}), we have 
$$\|G\|_{W_p^{\sigma}(\Omega)} \le C\rho^{m+d/2+d/p} h^{d/p-\sigma} \|\bfa\|_{\ell_p( \Xi)}
\le C\rho^{4m+3d} h^{d/p-\sigma} \|\bfa\|_{\ell_p( \Xi)}.$$ 

Taking the $L_p$ norm of $B$, we have 
$
\|\sum_{\xi\in\Xi} a_{\xi} (b_{\xi} - \chi_{\xi})\|_{W_p^{\sigma}(\Omega)}
\le
\max_{\xi\in \Xi} \bigl\| b_{\xi} - \chi_{\xi} \bigr\|_{W_p^{\sigma}(\Omega)}  \sum_{\xi\in\Xi} |a_{\xi}| 
$.
We control the $\ell_1$ norm by using H{\"o}lder's inequality $\|\a\|_1 \le (\#\Xi)^{\frac{p-1}{p}} \|\a\|_p $ and
 $\# \Xi\le C_{\Omega} \rho^d h^{-d}$.
Using Lemma \ref{ss_comparison} (or Lemma \ref{Matern_compare} in case $k_m = \kappa_m$),  we arrive at the desired inequality
\begin{equation}
\label{norm_diff_loc_full}
\|\sum a_{\xi} (b_{\xi} - \chi_{\xi})\|_{W_p^{\sigma}(\Omega)}
\le 
C \rho^{4m+3d}h^{J-d(\frac{p-1}{p})}
\|\bfa\|_{\ell_p(\Xi)} 
\le C \rho^{4m+3d}h^{d/p-\sigma}
\|\bfa\|_{\ell_p(\Xi)} 
\end{equation}
because the choice of $K$ ensures $J\ge d-\sigma$.
The theorem follows.
\end{proof}

%%%%%%%%%%%%%%%%%%%%%%%%%%%%%
%%
%%%%%%%%%%%%%%%%%%%%%%%%%%%%%%

For $s\in \V_{\Xi} = \spam_{\xi \in \Xi} b_{\xi}$ we may replace the discrete norm $\|\a\|_{\ell_p(\Xi)}$
by its equivalent $h^{-d/p}\|s\|_{L_p}$, as we now show.
%%%%%%%%%%%%%%%%%%%%%%%
%
%:Lower Comparison
%
%%%%%%%%%%%%%%%%%%%%%%%
\begin{proposition}\label{full_lowercomparison_local}{\bf(Local Basis Stability \& Nikolskii Inequality)} 
For every $\rho_0\ge 1$ there exists a constant  $h_0>0$ so that if $\Xi\subset \Omega$ has fill distance $h(\Xi,\Omega)\le h_0$
and mesh ratio $\rho\le \rho_0$, then
the family of local Lagrange functions generated with $K>\frac{10m-2}{\nu}$ satisfies the bounds
\begin{equation}
\label{lower_stability_estimate}
 c\left\| \bfa  \right \|_{\ell_p(\Xis)}       
\le  
 q^{-d/p} \| s \|_{L_p(\M)} \le  C\rho^{m+d/p}\left\| \bfa  \right \|_{\ell_p(\Xis)} 
\end{equation}
for all  $s = \sum_{\xi\in\Xis}a_{\xi}b_{\xi}\in \widetilde V_{\Xi}$, with $c= c(\rho, m, \Omega)$ and $C= C(\Omega,m)$.
In addition, for $1\le p,r\le \infty$, we have
\begin{equation}\label{local_Nikolskii_ineq}
\|s\|_{L_p(\Omega)} \le C q^{-d(\frac{1}{r}-\frac{1}{p})_+} \|s\|_{L_r(\Omega)}
\end{equation}
with $C=C(p,r,\rho,m,\Omega,)$.
\end{proposition}
%%%%%%%%%%%%%%%%%%%%%%%%%
%%%%%%%%%%%%%%%%%%%%%%%%%
\begin{proof}
The upper bound follows from the previous theorem, with $\sigma=0$. To obtain the lower bound, note that $q^{-d/p} \|s\|_{L_p(\Omega)} =
q^{-d/p}  \|\sum_{\xi\in\Xi} a_{\xi} (b_{\xi} -\chi_{\xi}) + \sum_{\xi\in\Xi} a_{\xi} \chi_{\xi}\|_{L_p(\Omega)}$. Consequently, by (\ref{compatible}), \eqref{full_riesz_ineq} and \eqref{norm_diff_loc_full}, we have
\[
q^{-d/p} \|s\|_{L_p(\Omega)} 
\ge 
q^{-d/p} \left(\|\textstyle{\sum_{\xi\in\Xi}} a_{\xi} \chi_{\xi}\|_{L_p(\Omega)} -   C \rho^{4m+3d}h^{J-d(\frac{p-1}{p})}\|\bfa\|_{\ell_p(\Xi)}\right) 
\ge 
(c_1 -  C \rho^{4m+3d}h^{J-d}) \|\bfa\|_{\ell_p(\Xi)}
\] 
where $c_1=c_1(\rho,\Omega,m)$ is the constant from Proposition \ref{full_lowercomparison}. Let $h_0>0$ be such that $c_1 -  C \rho^{4m+3d}h_0^{J-d}\ge \frac12 c_1$. This guarantees the same holds for all $0<h\le h_0$.
The proof of the Nikolskii inequality is, \emph{mutatis mutandis}, that of Corollary~\ref{TPS_Nikolskii}.
\end{proof}

%%%%%%%%%%%%%%%%%%%%%
%
%
%
%
%%%%%%%%%%%%%%%%%%%%%
\section{Main results and corollaries}\label{main_results}
At this point we can prove the inverse inequality for local  Lagrange functions in $\V_{\Xi}$:

\begin{theorem}\label{main}
Suppose $\Omega\subset\reals^d$ is a bounded Lipschitz region.
For $m>d/2$
and for every $\rho_0>0$ there exists a constant $h_0>0$, 
so that if $\Xi\subset \Omega$ has mesh ratio $\rho\le \rho_0$,  fill distance $h\le h_0$, 
and if $\X\subset \Om$ is a suitable extension of 
$\Xi$ 
(as mentioned in Remark \ref{True_Extension}) 
then 
for all $s = \sum_{\xi\in\Xi}a_{\xi}b_{\xi}\in \V_{\Xi}$ the following holds.  
For  $1\le p< \infty$ and all $0\le \sigma \le m-(d/2-d/p)_+$, 
or for $p= \infty$ and an integer $\sigma < m-d/2$, we have
$$
\left\|s\right\|_{W_p^{\sigma}({\Omega})} 
\le   
C   h^{-\sigma} \|s\|_{L_p(\M)}
$$
with $C = C(m, \rho, \Omega)$.
\end{theorem}
\begin{proof}
This is an immediate combination of Theorem~\ref{main_local_bernstein} and Proposition~\ref{full_lowercomparison_local}.
\end{proof}

\subsection{Restriction to the boundary}
Immediate applications of Theorem \ref{main} are the following ``trace'' estimates.
(Since the elements of $\V_{\Xi}$ are continuous, it is appropriate to consider these results about
 restriction to the boundary.)
To make sense of these, we first need to describe Sobolev spaces on the boundary $\partial \Omega$. 

%%%%%%%%%%
%
%
%
%%%%%%%%%%
\subsubsection{Smoothness spaces on $\partial \Omega$}\label{SS:Sobolev_boundary}
We use  the common tactic of employing
 a partition of unity with corresponding changes of variable to flatten the boundary. 
 (As in \cite[1.11]{Trieb2} and \cite{MaMi}, for instance.) 
 The details of the partition of unity and change of variable depends on the smoothness of 
 the boundary, and this influences the types of Sobolev spaces we can define (namely, the
 maximum order of smoothness is governed by the smoothness of the boundary).

For a domain whose boundary is Lipschitz
we consider
 a partition of unity $(\psi_j)_{j=1}^N$ of $\partial \Omega$, where
 each $\psi_j:\partial \Omega \to [0,1]$ is Lipschitz, and let 
$(U_j,h_j)_{j=1}^N$ be a corresponding collection
of bilipschitz {\em charts}  so that each $U_j$ is an open set in $\partial \Omega$ 
containing the closure of 
$\supp {\psi_j}$ and each $h_j:U_j \to \mathcal{O}_j\subset\reals^{d-1}$ is a 
bijective Lipschitz function. 
%
%
%
%\begin{definition}\label{Sobolev}
Then for $1\leq p<\infty$ 
and $0<\sigma\le 1$,
 the Sobolev space 
$W_p^{\sigma} \bigl(\partial \Omega \bigr)$ 
consists of functions $f\in L_p(\partial \Omega)$ such that
\begin{equation} \label{boundary_sobolev}
    \|f\|_{W_p^{\sigma}\bigl(\partial \Omega \bigr)}^p            %\|f\|_{W_p^k\bigl(SO(3)\bigr)}^p
:=\sum_{j=1}^N \|\bigl(\psi_j\circ (h_j^{-1})\bigr)\bigl(f\circ (h_j^{-1})\bigr)\|_{W_p^{\sigma}(\mathcal{O}_j)}^p
\end{equation}
is finite.
%\end{definition}

For higher orders of smoothness, we simply increase the smoothness of the boundary, and the partition of unity and chart.
For $\sigma<M$,
let $(\psi_j)_{j=1}^N$ be a $C^M$ partition of unity of $\partial \Omega$, and let 
$(U_j,h_j)_{j=1}^N$ be a  collection $C^M$ charts. Then $W_p^{\sigma}(\partial \Omega)$ consists of 
functions for which the  norm (\ref{boundary_sobolev}) is finite.

We note that this transporting of norms from Euclidean space to manifold by way of partition of unity and pull-back can be carried out for other smoothness spaces. In particular, it holds as well for the Besov class (see  again \cite{Trieb2} and \cite{MaMi}).
For this reason, it follows that for fractional $\sigma$, 
$W_p^{\sigma} \bigl(\partial \Omega \bigr) = B_{p,p}^{\sigma}(\partial \Omega)$ with equivalence of norms (as in the Euclidean case).

\subsection{Trace estimates}\label{SS_trace}
We may use Theorem \ref{main} to obtain the following trace estimate for functions in 
$\V_{\Xi}$. This is non-standard because the norms of the trace are bounded by $L_P$ norms rather than Sobolev norms.

\begin{corollary}\label{Trace} 
Under hypotheses of Theorem \ref{main}, for $s \in \V_{\Xi}$ we have, 
for $1\le p\le \infty$ and  $0<\sigma \le  m-1/p-(d/2-d/p)_+$, 
$$\| s\|_{W_p^{\sigma}(\partial \Omega)} 
\le 
C h^{-\sigma-1/p} \| s\|_{L_p (\Omega)}$$
with $C=C(m,\rho,\Omega)$.
\end{corollary}
\begin{proof}
For $\sigma>0$ we have that  $W_p^{\sigma+1/p}( \Omega) = B_{p,p}^{\sigma+1/p}(\Omega)$ and  
 $W_p^\sigma(\partial \Omega) = B_{p,p}^{\sigma}(\partial \Omega)$. It follows that
  $\mathrm{Tr}:W_p^{\sigma+1/p}( \Omega) \to W_p^\sigma(\partial \Omega)$
   is bounded
by the trace theorem 
(one will find a suitable one for smooth boundaries in
\cite[3.3.3]{Trieb1},
and for Lipschitz boundaries in \cite[Theorem 2.1]{MaMi})
 so 
$$\| s\|_{W_p^{\sigma}(\partial \Omega)} 
\le 
C_{\Omega}\| s\|_{W_p^{\sigma+{1/p}}( \Omega)} 
\le 
C_{\Omega,\rho} h^{-\sigma-1/p} \| s\|_{L_p (\Omega)}$$
The first inequality is from the trace theorem, while the second follows from Lemma \ref{lowercomparison}. 
\end{proof}
We can  get a similar estimate for $\sigma=0$, although this requires a modified trace result.
\begin{lemma}\label{ext_tr}
Suppose $\Omega$ is compact with $C^1$ boundary.
For $1<p<\infty$   there is a constant $C_{p}$ so that for all $u\in C^1(\overline{\Omega})$ and $\epsilon>0$  we have
$$\|u\|_{L_p(\partial {\Omega})}\le   C_{p}( \epsilon^{-q/p} \|u\|_{L_p(\Omega)}^p + \epsilon \|u\|_{W_p^1(\Omega)}^p)$$
with 
$q=\frac{p}{p-1}$ .
\end{lemma}
\begin{proof}
Note that in this case, we consider Sobolev norms over $\Omega$, so for the norms on the right hand side, we make use of the definition given in Section \ref{SS:Sobolev}. The $L_p(\partial \Omega)$ norm on the left is with respect to surface measure, but this can be estimated in a standard way  (by partition of unity and change of variables).

We begin by proving a trace result for $\Omega=\reals_+^d = \reals^{d-1}\times[0,\infty)$.
For $u\in C^1(\reals_+^d)$ having compact support
 and $x'\in \reals^{d-1}$, let $r_{x'}$ be
the first positive zero of $t\mapsto |u(x',t)|$. Then
\begin{eqnarray*}
 \bigl|u(x',0)\bigr|^p
&\le& 
\int_0^{r_{x'}} \partial_{d} |u(x',x_d)|^p d x_d\\
&\le&
p\int_0^{\infty}|u(x',x_d)|^{p-1}\, | \partial_{d} (u(x',x_d)) | d x_d\\
&\le& 
\int_0^{\infty}C(\epsilon) |u(x',x_d)|^{(p-1)\frac{p}{p-1}}+ \epsilon| \partial_{d} (u(x',x_d))|^p d x_d
\end{eqnarray*}
The last line uses Young's inequality $ab\le C(\epsilon) a^q + \epsilon b^p$ with 
$C(\epsilon)= q^{-1}(\epsilon p)^{-q/p}$.
Integrating this over $\reals^{d-1}$, we have
\begin{equation}\label{modified_trace}
\|u\|_{L_p(\reals^{d-1})}^p \le C_{p} \epsilon^{-q/p} \|u\|_{L_p(\reals_+^d)}^p + \epsilon \| u\|_{W_p^1(\reals_+^d)}^p.
\end{equation}

%For $f:\Omega\to \R$, 
Now let $(\Psi_j)_{j=1}^N$ be a finite collection of non-negative, compactly supported, $C^1$ functions 
so that $\sum \Psi_j =1$ in a small neighborhood of $\partial \Omega$.
Let $(U_j)_{j=1}^N$ denote a corresponding collection of open sets 
so that $\supp \Psi_j \subset U_j$ and so that there is $h_j:U_j \to B(0,\epsilon_j)$, an open ball in $\reals^d$.

For $f\in C^1(\overline{\Omega})$ and $1\le j\le N$, the product $\Psi_j f$ is compactly supported and (extending by 0) we have
$u_j := (\Psi_j f)\circ (h_j^{-1}) \in C_1(\reals^d_+)$.
Applying (\ref{modified_trace}) to $(\Psi_j f)\circ h_j^{-1}$ gives 
$
\|u_j\|_{L_p(\reals^{d-1})}^p \le C_p \epsilon^{-q/p} \|u_j\|_{L_p(\reals^d)}^p + \epsilon \| u_j\|_{W_p^1(\reals^d)}^p.
$
Because $\Psi_j$ and $h_j^{-1}$ are $C^1$ over compact sets, their  norms
can be bounded independent of $j$. By applying chain and product rules, it follows that 
\begin{equation}\label{prelim_mod_trace}
\sum_{j=1}^N \|u_j\|_{L_p(\reals^{d-1})}^p \le C_{p,q} \bigl(\epsilon^{-q/p} \|u\|_{L_p(\Omega)}^p + \epsilon \| u\|_{W_p^1(\Omega)}^p\bigr).
\end{equation}
with an increased 
constant which depends on that of (\ref{modified_trace})  
as well as  
$\max_{1\le j\le N} \|(h_j)^{-1}\|_{C_1\bigl(h_j^{-1}(\supp{\Psi_j})^- \bigr)}$
and 
$\max_{1\le j\le N} \|\Psi_j\|_{C_1(\reals^d)}$. 
Because $\left({\Psi_j}_{\left|_{\partial \Omega}\right.}\right)$ is a partition of unity for $\partial \Omega$, 
the left hand side of (\ref{prelim_mod_trace}) 
controls the $L_p$ norm of $u_{\left|_{\partial \Omega}\right.}$,
which gives the $\epsilon$-modified trace inequality
$$\|u\|_{L_p(\partial \Omega)}^p \le C_{p,q}(  \epsilon^{-q/p} \|u\|_{L_p(\Omega)}^p + \epsilon \|u\|_{W_p^1(\Omega)}^p).$$
\end{proof}

\begin{corollary}\label{Modified_Trace} 
Let $\Omega$ be a bounded domain with  $C^1$ boundary
and assume the hypotheses of Theorem \ref{main}.
For $s \in \V_{\Xi}$ we have, 
for $1\le p\le \infty$ and  $1 + (d/2-d/p)_+\le m$
that
$$\| s\|_{L_p(\partial \Omega)} 
\le 
Ch^{-1/p} \| s\|_{L_p (\Omega)}$$
with $C=C(p,\rho,m,\Omega)$.
\end{corollary}
\begin{proof}
%We begin by getting the result for $p\ge 2$.
For $p=1$ the Theorem follows directly from the 
boundedness of trace from $W_1^1(\Omega)$ to $L_1(\partial \Omega)$  (see \cite[Theorem 1 Chapter 5.5]{evans})
and by repeating the argument of Theorem \ref{lowercomparison}.

For $1<p<\infty$, we apply Lemma \ref{ext_tr} with  $\epsilon = h^{p-1}$ (so that $\epsilon^{-q/p} = h^{-\frac{p-1}{p-1}}$) 
followed by Theorem \ref{lowercomparison}. Thus,
\begin{eqnarray*}
\|s\|_{L_p(\partial \Omega)}^p 
%&\le& C(\epsilon^{-q/p} \|s\|_{L_p(\Omega)}^p + \epsilon \|s\|_{W_p^1(\Omega)}^p)\\
&\le& C(h^{-1} \|s\|_{L_p(\Omega)}^p + h^{p-1} \|s\|_{W_p^1(\Omega)}^p)\\
&\le& C(h^{-1} \|s\|_{L_p(\Omega)}^p + h^{-1} \|s\|_{L_p(\Omega)}^p)
\end{eqnarray*}
and the result follows by taking the $p$th root.
\end{proof}

\appendix
\section{Energy and pointwise bounds on the Lagrange function}\label{Appendix_A}

In this section, we show that Lagrange functions for  surface splines and 
Mat{\' e}rn kernels satisfy decay estimates as in Section \ref{SS:Lagrange}. 
 
We say that $\Omega$ satisfies an interior cone condition if
there are constants $\varphi\in (0,\pi/2)$ and $0<R<\infty$ so that 
  for every $x\in \Omega$ there is a  cone 
  $\mathcal{C}_{\vec{n}} = \{y \mid |y - x| \le R, n\cdot(\frac{y-x}{|y-x|})\ge \cos \varphi\}$ opening in the direction 
  determined by the unit vector $\vec{n}$
  so that $\mathcal{C}_{\vec{n}}\subset \Omega$.

 We recall the zeros estimate 
 \cite[Theorem A.11]{HNW-p}
 for a bounded region $\Omega$ with Lipschitz boundary
 (the version we cite is a streamlined modification of an earlier estimate given in \cite[Theorem 2.12]{NWW}).
\begin{lemma}[{\bf Zeros estimate}] \label{zl} 
Let $1\le p\le \infty$ and  $m>d/p$ (when $p=1$ we may have $m\ge d/p$).
Suppose $\Omega$  satisfies a cone condition with aperture $\varphi$ and radius $R$. 
Then there are constants $h_1$ 
(depending on $m$ and $\varphi$) and $\Lam$ (depending on $m,d,p, \varphi$) so that 
if $X\subset \Omega$  has fill distance $h\le h_1 R$ 
and if $u\in W_p^{m}(\Omega)$ satisfies $u\left|_{X}\right. = 0$ then 
 \[
 \|u\|_{W_p^{k}(\Omega)}\le  \Lam h^{m- k} \|u \|_{W_p^{m}(\Omega)}
\]
and
 \[
 \|u\|_{L_{\infty}(\Omega)}\le  \Lam h^{m- d/p} \|u \|_{W_p^{m}(\Omega)}.
\]
\end{lemma}
An important feature of this lemma is that the density $h$ is controlled by the cone radius $R$, but the constant $\Lam$ does not depend on $R$. 
This allows a comparison of results across  sets which are geometrically related.
E.g.,  annuli $B(x,r_2)\setminus B(x,r_1)$
satisfy cone conditions with aperture $\varphi$ independent of $r_2$ and $r_1$, and with cone radius
equal to half the thickness $\frac{r_2-r_1}{2}$.
 Thus, the above result holds for any point set with $h \le h_1(r_2-r_1)/2$. 
With almost no modification, this result extends to balls $B(x,r)$ (where $h \le h_1 r$) and complements of balls (where
there is not restriction on $h$).

Consider now the annulus $\a(\xi,r,t):=\{x\in \reals^d\mid r-t<|x-\xi|\le r\}$.  
Applying Lemma \ref{zl} with $p=2$ and $k=m-1$,
 we  estimate the Sobolev norm\footnote{Recall that we use the Sobolev norm as defined in Section \ref{SS:Sobolev} -- in 
 particular, the $kth$ order  partial derivatives are weighted by $\begin{pmatrix}m\\k \end{pmatrix}$.}
  as
$\|u\|_{W_2^{m}(\a)}^2
\le 
|u|_{W_2^{m}(\a)}^2 +
m \Lam^2 h^{2} \|u \|_{W_2^{m-1}(\a)}^2 $
 which, after rearranging terms, implies  that
$\|u\|_{W_2^{m}(\a)}^2\le 
\frac{1}{1- m \Lam^2 h^{2}} |u|_{W_2^{m}(\a)}^2 $
for $u$ vanishing on $X\subset \a$ with $h\le h_1t/2$.
In short, if 
$h\le \min\bigl(\frac{h_1 t}2,h_2\bigr)$ with 
\begin{equation}\label{h_2}
h_2 := (\sqrt{2m}\Lam)^{-1}\end{equation}
then
\begin{equation}\label{zeros_annulus_seminorm}
 |u|_{W_2^{k}(\a)}
 \le 
 \|u\|_{W_2^{k}(\a)}
 \le
 \Lam h^{m-k} \|u \|_{W_2^m(\a)}
 \le 
 2\Lam h^{m-k} |u |_{W_2^m(\a)}
\end{equation}
for $u$ vanishing on $X$.

\begin{lemma}\label{one_step_energy}
Suppose $m>d/2$. 
There is a constant $\nu= \nu(m,d)$ with $\nu<1$ 
such that if $X \subset \reals^d$ is a finite point set, 
$\a = \a(\xi,r,t)$ is the annulus of outer radius $r$, width $t$ and center $\xi\in X$,
and  $X_0= X\cap \a$ has fill distance
$h = h(X_0,\a)
\le
 \min\bigl( \frac{h_1 t}{2}, h_2\bigr) 
 $, 
then% for $\xi \in \Xi := \widetilde{\Xi}\cap \Omega$ and $r\le \diam(\Omega)$,
\begin{itemize}
\item  the Mat{\' e}rn Lagrange function 
$\chi_{\xi} \in \spam{\{\kappa_m(\cdot  - \zeta)\mid \zeta \in X\}}$
satisfies
$$\|\chi_{\xi} \|_{W_2^m\bigl(\reals^d\setminus B(\xi, r)\bigr)} \le  \nu \| \chi_{\xi} \|_{W_2^m\bigl(\reals^d\setminus B(\xi, r-t)\bigr)}.$$
 \item the   Lagrange function  
$\chi_{\xi} \in S(\phi_m,X)$ 
for the order $m$ surface spline
satisfies
$$|\chi_{\xi} |_{W_2^m\bigl(\reals^d\setminus B(\xi, r)\bigr)} \le  \nu | \chi_{\xi} |_{W_2^m\bigl(\reals^d\setminus B(\xi, r-t)\bigr)}.$$
\end{itemize}
\end{lemma}
\begin{proof}
In either case, the function $k_m$ is the reproducing kernel for a (semi-)Hilbert space 
(described in Sections \ref{SS:M_kernels} and \ref{SS:ss_kernels}), 
and we use the notation $[u]_{m,Y}$ to denote $\|u\|_{W_2^m(Y)}$ or $|u|_{W_2^m(Y)}$, respectively.

Let $\tau:\reals \to [0,1]$  
be a smooth cut-off function supported on the interval $(-\infty,1)$  equaling $1$ on $(-\infty, 0]$.
For $r>t$, we define $\tau_{\xi,r,t}:\reals^d\to \reals$ as $\tau_{\xi,r,t} (x) = \tau\bigl(\frac{1}{t}(|x-\xi| - (r-t)\bigr)$, 
and note that it is a smooth function
supported in the ball $B(\xi,r)$, and equals 1 in $B(\xi,r-t)$.  
By the chain rule, there is a bound $\|D^{ \beta} \tau_{\xi,r,t}\|_{\infty} \le C t^{-| \beta|} $ 
which is independent of $r$.

Both $\chi_{\xi}$ and $\tau_{\xi, r,t} \chi_{\xi}$ are Lagrange functions on 
$X$. %(they both interpolation $\delta_{\xi}$).
Thus 
$[\chi_{\xi}]_m%_{W_2^m(\reals^d)} 
\le 
[\tau_{\xi,r,t} \chi_{\xi}]_m
%|_{W_2^m(\reals^d)}
$.
Using the additivity  of $[\cdot]_m$, 
and noting that the Lagrange functions are identical on $B(\xi,r-t)$ 
while $\tau_{\xi,r,t}$ vanishes outside $B(\xi,r)$, we have
$$
[\chi_{\xi}]_m^2
\le 
[\tau_{\xi,r,t} \chi_{\xi}]_m^2
\qquad
\longrightarrow
\qquad
[\chi_{\xi}]_{m,\reals^d\setminus B(\xi,r-t)}^2
\le 
[\tau_{\xi,r,t} \chi_{\xi}]_{m, \a(\xi,r,t)}^2.
$$
By using H{\" o}lder's inequality in conjunction with the product rule, we have 
\begin{eqnarray}\label{Matveev_product_rule}
\int_{\a}|D^{\alpha} \bigl(\tau_{\xi,r,t}(x) \chi_{\xi}(x)\bigr) |^2 \dif x
&=&
\int_{\a}|
  \sum_{\beta\le \alpha} 
    \begin{pmatrix}\alpha \\ \beta\end{pmatrix}D^{\alpha - \beta}\tau_{\xi,r,t}(x) D^{ \beta} \chi_{\xi}(x) |^2
\dif x
\nonumber\\
&\le &
C
\sum_{\beta\le \alpha}
t^{-2|\alpha - \beta|}
\int_{\a}|D^{\beta} \chi_{\xi}(x)\bigr) |^2 \dif x   \nonumber \\
&\le &
C
\sum_{\beta\le \alpha}
h^{-2(|\alpha| -|\beta|)}
\int_{\a}|D^{\beta} \chi_{\xi}(x)\bigr) |^2 \dif x
\end{eqnarray}
In the last line we have used that $\beta \le \alpha$,  and thus
$t^{-|\alpha - \beta|} = t^{-|\alpha| + |\beta|} \le h_0^{|\alpha| - |\beta|}h^{-|\alpha| + |\beta|}$. 
Applying  
%Lemma \ref{zl} 
(\ref{zeros_annulus_seminorm})
to  (\ref{Matveev_product_rule})
gives, for each 
$\beta\le \alpha$, $\int_{\a}|D^{\beta} \chi_{\xi}(x)\bigr) |^2 \dif x\le C^2 h^{2(m-|\beta|)}\|\chi_{\xi}\|_{W_2^m(\a)}$.
This yields the inequality
$$
[\tau_{\xi, r,t} \chi_{\xi}]_{m,\reals^d\setminus B(\xi,r-t)}^2
\le 
C
\sum_{|\alpha| = m} 
\sum_{\beta\le \alpha}
\begin{pmatrix}\alpha \\ \beta\end{pmatrix}
\bigl(
h^{-2|\alpha| + 2|\beta|}\bigr)
C^2
h^{2(m-|\beta|)}|\chi_{\xi}|_{W_2^m(\a)}^2.
$$
Canceling powers of $h$ and collecting constants which depend only on $m$ and $d$, we have
$$
[ \chi_{\xi}]_{m,\reals^d\setminus B(\xi,r-t)}^2
\le
[\tau_{\xi, r,t} \chi_{\xi}]_{m,\reals^d\setminus B(\xi,r-t)}^2
\le 
C
[ \chi_{\xi}]_{m,\a}^2
$$
Finally, we note that 
$[ \chi_{\xi}]_{m,\a}^2= [ \chi_{\xi}]_{m,\reals^d \setminus B(\xi, r-t)}^2-[ \chi_{\xi}]_{m,\reals^d \setminus B(\xi, r)}^2$
which yields
$$[ \chi_{\xi}]_{m,\reals^d \setminus B(\xi, r)}^2 \le \frac{C-1}{C} [ \chi_{\xi}]_{m,\reals^d \setminus B(\xi, r-t)}^2  $$
and the lemma follows with $\nu = \sqrt{\frac{C-1}{C}}<1$.
\end{proof}

We may now iterate Lemma \ref{one_step_energy} to get control of the ``energy'' of the tail of the Lagrange functions.
\begin{lemma}\label{appendix_matveev}
Suppose $D\subset \reals^d$ is bounded, 
and $X\subset D$ is a finite point set with fill distance satisfying $h(X,D) \le h_2$. %h_0 =(\sqrt{2m}\Lam)^{-1}$.
There is $\mu = \mu(m,d)>0$ so that for $R<\dist(\xi, \partial D)$
\begin{itemize}
\item  the Mat{\' e}rn Lagrange function 
$\chi_{\xi} \in \spam{\{\kappa_m(\cdot  - \zeta)\mid \zeta \in X\}}$
satisfies
$$\|\chi_{\xi} \|_{W_2^m\bigl(\reals^d\setminus B(\xi, R)\bigr)} \le C q^{d/2 -m} \exp\left(-\mu \frac{R}{h}\right)$$
 \item the   Lagrange function  
$\chi_{\xi} \in S(\phi_m,X)$ 
for the order $m$ surface spline
satisfies
$$|\chi_{\xi} |_{W_2^m\bigl(\reals^d\setminus B(\xi, R)\bigr)} \le 
C q^{d/2 -m} \exp\left(-\mu \frac{R}{h}\right)
$$
\end{itemize}
\end{lemma}
\begin{proof}
Setting $t = 4h/h_1$ (where $h_1$ is the constant appearing in Lemma \ref{zl}), consider, for $r\le \dist(\xi,\partial D)$, an annulus $\a(\xi,r,t)$ and the restricted  point set $X_0 = X\cap \a(\xi,r,t)$.
The slightly smaller, inner annulus $\a(\xi,r-h,t-2h)$ has the property that for every $x\in \a(\xi,r-h,t-2h)$, there is $\zeta\in X_0$ so
that $\dist(x,\zeta) \le h$ (since in that case $\dist(x,X_0) = \dist(x,X)$).
It follows that $h(X_0, \a(\xi,r,t)) \le 2h$ and therefore $h(X_0, \a(\xi,r,t)) \le\frac{ h_1 t}{2}$. 

Now letting $n = \lfloor R/t\rfloor$,  by Lemma \ref{one_step_energy} we have that 
$$|\chi_{\xi} |_{W_2^m\bigl(\reals^d\setminus B(\xi, R)\bigr)} \le \nu |\chi_{\xi} |_{W_2^m\bigl(\reals^d\setminus B(\xi, R-t)\bigr)}
\le \dots \le \nu^n |\chi_{\xi} |_{W_2^m\bigl(\reals^d\bigr)}\le  \nu^{-1} \nu^{\frac{h_0R}{4h}}   |\chi_{\xi} |_{W_2^m\bigl(\reals^d\bigr)}.$$
By the ``bump estimate'' (\ref{bump_estimate}), we have that $|\chi_{\xi}|_{W_2^m\bigl(\reals^d\bigr)}\le C q^{d/2 -m}$, so the lemma follows
with $\mu = -\frac{h_0}{4} \log(\nu)$, which is positive since $\nu<1$.
\end{proof}

Note that if $\Omega\subset \reals^d$ is compact, then $\Om = \{x\in \reals^d\mid \dist(x,\Omega)\le \diam(\Omega)\}$ 
automatically satisfies a cone condition (with radius $R = \diam(\Omega)$ and aperture independent of $\Omega$). Thus, the result (\ref{lagrange_decay}) follows with $D  = \Om$, $X=\X$,  $h_0= h_0(d,m)$ and $\xi \in \Xi$.

Because $x\in \Omega$, $\xi\in \Xi$ implies $R=|\xi- x| \le \dist(\xi,\partial \Omega)$,
we can apply the second part of the zeros estimate 
Lemma \ref{zl} 
to Lemma \ref{appendix_matveev} to obtain
 the pointwise estimate (\ref{ptwise})
\begin{equation}
\label{bounded_lagrange_functions}
|\chi_{\xi}(x)| 
\le 
C h^{m-d/2} \| \chi_{\xi} \|_{W_2^m(\reals^d \setminus B(\xi,\dist(\xi,x)))} 
\le 
C \rho^{m-d/2} \mathrm{exp}\left(-\mu \frac{\dist(x,\xi)}{h}\right). %\ C=C(m,\Omega,d)
\end{equation}
Note that 
in the second inequality we have written $ h^{m-d/2} q^{d/2 -m} = \rho^{m-d/2}$.

\section{Stability bounds for the Lagrange function}\label{Appendix_B}

We now demonstrate that the family of Lagrange functions for suitable kernels over a domain $\Omega$ satisfy stability bounds of the form (\ref{full_riesz_ineq}).
This was demonstrated in \cite[Proposition 3.6 \& Theorem 3.7]{HNSW_2_2011}; we follow the argument presented there, with modifications for dealing with a suitably
bounded Euclidean domain, and to obtain a necessary refinement: that the threshold fill distance $h_0$ depends only on $m$ and $d$ (and not on $\rho$ or $\Omega$).

\begin{lemma}\label{uppercomparison}
Suppose $\Omega$ is a bounded domain and $\Xi\subset \Omega$ is a finite subset
with fill distance $h\le h_2$, where $h_2=h_2(m,d)$ is the constant given in (\ref{h_2}). 
There exists a constant $c_2= c_2(m,d)$ 
so that the family of functions
$(\chi_{\xi})_{\xi\in \Xi}$ have the property that
for $s = \sum_{\xi\in\Xi} a_{\xi} \chi_{\xi}$,
$$  \|s\|_p \le c_2  \rho^{m+d/2} q^{d/p} \| \bfa \|_{\ell_p(\Xi)}$$
holds.
\end{lemma}
%%%%%%%%%%%%%%%%%%%%%%%%%
%Proof of Upper Bound
%%%%%%%%%%%%%%%%%%%%%%%%%
\begin{proof}
For $p=\infty$, 
inequality (\ref{bounded_lagrange_functions}) leads to a bound on the Lebesgue constant 
$\mathcal{L}$ for the $\chi_\xi$'s over $\Omega$: 
\begin{equation}\label{lebesgue_constant}
\mathcal{L}:= \sup_{x\in \Omega} \big(\textstyle{\sum_{\xi \in \Xi}} |\chi_{\xi}(x)| \big) < C\rho^{m+d/2},\ C=C(m,d).
\end{equation}
Indeed, for fixed $x\in \Omega$ we note that
$
\sum_{\xi \in \Xi} |\chi_{\xi}(x)|  \le C \rho^{m-d/2}\sum_{\xi \in \Xi}e^{-\mu \frac{\dist(x,\xi)}{h}}.
$
By estimating en annuli, using sets 
$A_{n}:=\{\xi\in \Xi \mid h(n-1)< |x-\xi| \le hn\} $ having $\#A_n \le C \left(\frac{hn}{q}\right)^{d} $,
we have that
\[
\sum_{\xi \in \Xi} |\chi_{\xi}(x)|  \le C \rho^{m-d/2}\big( 1+ \sum_{n=1}^\infty \rho^d n^{d}  e^{-\mu n}\big) < C \rho^{m+d/2},
\]
where the constant $C=C(m,d)$ is independent of $\Xi$. Taking the supremum then yields \eqref{lebesgue_constant}.
It follows  that
$$
\|s\|_{L_{\infty}(\M)}\le \mathcal{L}\|s|_{\Xi}\|_{\ell_{\infty}(\Xi)} 
=
\mathcal{L}\|\bfa\|_{\ell_{\infty}(\Xi)},
$$

For $p=1$, we have
$$
\int_{\Omega} |s(x)|\dif x 
\le 
\sum_{\xi\in\Xi} |a_{\xi}|  \int_{\Omega}|\chi_{\xi}(x)|\dif x
\le 
Ch^{d} \|\bfa \|_{\ell_{1}(\Xi)}.
$$
Here we have used the fact that $\|\chi_{\xi}\|_1\le C \rho^{m-d/2}h^{d} \le C \rho^{m+d/2}q^{d}$, which follows by integrating (\ref{bounded_lagrange_functions}).
A standard application of operator interpolation proves the other cases. 
\end{proof}

\subsubsection*{Preliminary estimates}
Because $\Omega$ satisfies a cone condition with aperture $\varphi$ and radius $R$, there is a constant $\alpha$ (depending only on $d$ and  $\varphi$) so that for all $x\in \Omega$,
$$ \alpha r^d \le \vol(B(x,r)\cap \Omega)$$
for $r\le R$ (the radius of the cone condition). 
Similarly, we have that there is a constant $K$ (depending only on $d$ and $\varphi$) so that
\begin{equation}\label{pointcount}\#\bigl(\Xi\cap B(x,r)\bigr)\le K (r/q)^d.\end{equation} 

We can use a simple modification of the zeros lemma \cite[Lemma 7.1]{HNSW} valid for balls, 
which states that there exists a constant $h_3>0$ depending only on $m$ and $d$ 
so that for $h\le h_3$,  the H{\"o}lder-like condition
$$|\chi_{\xi}(x) - \chi_{\xi}(y)|\le C \rho^{m-d/2} \left(\frac{|x-y|}{q}\right)^{\epsilon} $$
holds  for $0<\epsilon <m -d/2$ and with a constant $C=C(d,m)$.

\begin{remark} For the remainder of this appendix, we assume $\Xi$ is sufficiently dense that $h(\Xi,\Omega)\le h_0:=\min(h_2,h_3)$.
We note that $h_0$ depends only on $d$ and $m$ (because this is true for $h_2$ and $h_3$).
\end{remark}

This permits us to understand the structure of $\chi_{\xi}$ around the centers $\zeta\in \Xi$.
Namely, because $\chi_{\xi}(\xi)=1$,
 $$\chi_{\xi}(x)\ge \frac{2}{3}\quad \text{ for $x$ in } B(\xi,\gamma q)$$ 
whenever 
$\gamma^{\epsilon} \le1/(3 C\rho^{m-d/2})$.
For the off-center case (i.e., when $\zeta\neq \xi$),
\begin{equation}\label{Holder_zero}
 |\chi_{\zeta}(x)|\le C \rho^{m-d/2}  \gamma^{\epsilon}, \quad \text{for $x$ in }B(\xi,\gamma q)
\end{equation}
with a constant $C=C(d,m)$.

Now fix $0<\gamma\le 1/(3 C\rho^{m-d/2})^{1/\epsilon}$ and define $B_{\xi} := \Omega \cap B(\xi,\gamma q)$. The above estimate
guarantees that $ \alpha (\gamma q)^d \le \vol(B_{\xi})$, and 
$$
\alpha (\gamma q)^d \left( \frac23\right)^p \le  \int_{B_{\xi}} |\chi_{\xi}(x) |^p \dif x 
\quad
\Longrightarrow
\quad
 \alpha (\gamma q)^d \left( \frac23\right)^p \sum_{\xi\in \Xi} |a_{\xi}|^p 
 \le 
 \sum_{\xi\in \Xi}\int_{B_{\xi}} |a_{\xi}\chi_{\xi}(x)|^p \dif x.  
 $$
This is the starting point for the corresponding lower bound to Lemma \ref{uppercomparison}, since the quasi-triangle inequality
$(A+B)^p \le 2^{p-1} (A^p+B^p)$ implies that
$ 
|a_{\xi}\chi_{\xi}(x)|^p 
\le 
2^{p-1}
\left(|\sum_{\zeta\in \Xi}  a_{\zeta}\chi_{\zeta}(x)|^p +  |\sum_{\zeta\neq \xi}  a_{\zeta}\chi_{\zeta}(x)|^p \right)
$ 
and so
\begin{equation}\label{overshot}
\alpha (\gamma q)^d \left( \frac23\right)^p \sum_{\xi\in \Xi} |a_{\xi}|^p 
\le
2^{p-1} 
\sum_{\xi\in \Xi}
  \int_{B_{\xi}}  
    \left(\left|\sum_{\zeta\in \Xi}  a_{\zeta}\chi_{\zeta}(x)\right|^p +  \left|\sum_{\zeta\neq \xi}  a_{\zeta}\chi_{\zeta}(x)\right|^p\right)
  \dif x.
\end{equation}
The desired lower bound is 
$
c_1 q^{d/p}\| \bfa   \|_{\ell_p(\Xi)}       
\le  
\|s\|_p
$,
so we must estimate the size of the overestimated ``off-diagonal'' terms:
$2^{p-1}\sum_{\xi\in \Xi}
  \int_{B_{\xi}}|\sum_{\zeta\neq \xi}  a_{\zeta}\chi_{\zeta}(x)|^p
  \dif x$.
  \subsubsection*{Controlling the off-diagonal terms}
 This is done in two stages, by splitting  
 $
\sum_{\xi\in \Xi}
  \int_{B_{\xi}}
  \left|\sum_{\{\zeta\in \Xi :\zeta\neq \xi \}}  a_{\zeta}\chi_{\zeta}(x)\right|^p
  \dif x = \sum_{\xi\in \Xi}(I_{\xi}+II_{\xi})$
  where the first term is the ``far away'' contribution
  $I_{\xi}:=
  \int_{B_{\xi}}
  \left|\sum_{\{\zeta\in \Xi :\ |\zeta-\xi|\ge \Gamma q\}}  a_{\zeta}\chi_{\zeta}(x)\right|^p
  \dif x$ 
  and the second is the nearby contribution
  $II_{\xi} :=
  \int_{B_{\xi}}
  \left|\sum_{\{\zeta\in \Xi :\ |\zeta-\xi|\le \Gamma q\}}  a_{\zeta}\chi_{\zeta}(x)\right|^p
  \dif x$. These depend on an (as yet) undetermined parameter $\Gamma>0$.
 First we use the exponential decay of (\ref{bounded_lagrange_functions}) to control the far away portion of the off-diagonal part,
 Then we use the H{\"o}lder estimates (\ref{Holder_zero}) to bound the nearby portion.
\begin{lemma}\label{farawaypart}
For every $p\in [1,\infty)$ there is a function $F:(0,\infty) \to \reals$ satisfying $\lim_{t\to \infty}F(t) = 0$ so that
for every $\Gamma>0$, we have the inequality
$$
\sum_{\xi\in \Xi}
  \int_{B_{\xi}}
  \left|\sum_{\{\zeta\in \Xi :\ |\zeta-\xi|\ge \Gamma q\}}  a_{\zeta}\chi_{\zeta}(x)\right|^p
  \dif x
  \le 
F(\Gamma)  (\gamma q)^d  \rho^{p(m-\tfrac{d}2)} 
     \|\a\|_{\ell_p(\Xi)}^p$$
  holds with $F(\Gamma) \le \tilde{C} e^{-\frac{\mu}{2} p \Gamma}$ with $\tilde{C} = \tilde{C} (m,d,p)$ and $\mu=\mu(d,m)$ the constant from (\ref{lagrange_decay}).
\end{lemma}  
\begin{proof}
We sum over the non-overlapping dyadic regions
%%%%%%%%%%%%%%%%%%%%
%Dyadic Regions
%%%%%%%%%%%%%%%%%%%%
$$
\Omega_k
:=
\Omega_k(\xi)
:=
\{
  \zeta\in\Xi 
  \mid 
  \Gamma 2^k q \le \d(\xi,\zeta) \le \Gamma 2^{k+1} q
\},\quad k=0, 1, \ldots, N_{q},
$$ 
%%%%%%%%%%%%%%%%%%
where
$2^{N_q}\sim \frac{\diam (\Omega)}{\Gamma q}$. 
This means that, for
%%%%%%%%%%%%%%%%%%%%%%
%M_k Definition
%%%%%%%%%%%%%%%%%%%%%%
$M_k
:=
\int_{B_\xi}
  \left|
    \sum_{\zeta\in \Omega_k} 
      a_\zeta \chi_{\zeta}(x)
  \right|^p 
\dif x, 
$
%%%%%%%%%%%%%%%%%%%%%%%
$$
I_{\xi}
\le  
\sum_{k=0}^{N_{q} }
2^{(p-1)(k+1)}
  \int_{B_\xi}
    \left|\sum_{\zeta \in \Omega_k} a_\zeta \chi_{\zeta}(x)\right|^p 
  \dif x 
= 
\sum_{k=0}^{N_q} 2^{(p-1)(k+1)} M_k,
$$
where the above inequality follows by iterating the quasi-triangle inequality
$|A+B|^p \le 2^{p-1}(|A|^p+|B|^p)$ to get
$
 \left|\sum^n_{j=1} A_j\right|^p \le \sum^n_{j=1} 2^{j(p-1)} |A_j|^p.
$

We now estimate the contribution from each $M_k$, the portion of $II_{\xi}$ 
coming from the dyadic interval $\Omega_k$.
By using the (generalized quasi-triangle) inequality $ |\sum_{j=1}^n A_j|^p\le n^{p-1}\sum |A_j|^p$, we have
\begin{align}
M_k &\le 
\left(\# \Omega_k\right)^{p-1} \times
\sum_{\zeta\in  \Omega_k} \int_{B_\xi} |a_\zeta \chi_{\zeta}(x)|^p \dif x\nonumber\\
&\le
\left(\# \Omega_k\right)^{p-1}\times \max_{\zeta\in \Omega_k} \|\chi_{\zeta}\|_{L_{1}(B_\xi)}^p
\times \sum_{\zeta\in\Omega_k} |a_\zeta|^p \nonumber\\
&\le 
\left(K (2^{k+1}\Gamma)^d\right)^{p-1} 
C (\gamma q)^d 
\rho^{p(m-\tfrac{d}2)} 
\left(  \exp(-\mu p\Gamma  2^k) \right)
\sum_{\zeta\in \Omega_k}|a_\zeta|^p.\label{holder_trick}
\end{align}
In the final line, we have used the estimates (\ref{pointcount}) %$\#\Omega_k\le K_{\M} (2^{k+1} \Gamma)^d$,
and (\ref{bounded_lagrange_functions}).

Multiplying by $2^{(p-1)(k+1)}$ and summing from $0$ to $N_q$, we obtain
(after rearranging some terms and combining constants which depend only on $d$ and $p$)
$$
I_{\xi}
\le
C
(\gamma q)^d  
\rho^{p(m-\tfrac{d}2)} 
\left( 
  \sum_{k=0}^{N_q}  
    (2^{k(d+1)}\Gamma^{d})^{p-1}
    \exp(-\mu p \Gamma  2^k)  
     \left[
       \sum_{\zeta\in \Omega_k}|a_\zeta|^p
     \right] 
\right).
$$
We can now sum over $\xi$, obtaining
\begin{align*}
\sum_{\xi\in \Xi} I_{\xi}
&\le
C
 (\gamma q)^d  
 \rho^{p(m-\tfrac{d}2)} 
\left( 
  \sum_{k=0}^{N_q} 
    (2^{k(d+1)}\Gamma^d)^{p-1} 
    \exp(-\mu p \Gamma  2^k)  
    \sum_{\xi\in\Xi}
     \left[
       \sum_{\zeta\in \Omega_k}|a_\zeta|^p
     \right] 
\right)\\
&\le
%\mathcal{K}
C
 (\gamma q)^d  
 \rho^{p(m-\tfrac{d}2)} 
\left( 
  \sum_{k=0}^{N_q} 
    (2^{k(d+1)}\Gamma^d)^{p-1} 
    \exp(-\mu p \Gamma  2^k)
    (K (2^{k+1}\Gamma)^{d})  
     \left[
       \sum_{\zeta\in \Xi}|a_\zeta|^p
     \right] 
\right)\\
&\le
%\mathcal{K}
C
\frac{(\gamma q)^d \rho^{p(m-\tfrac{d}2)} }{\Gamma^p}
\left( 
  \sum_{k=0}^{N_q}  
    (2^{k}\Gamma)^{(d+1)p} 
    \exp(-\mu p \Gamma  2^k)  
    \right) \left[
     \sum_{\zeta\in\Xi}
         |a_\zeta|^p
     \right].
 \end{align*}
In the second inequality, we have exchanged summation over $\xi$ and $\zeta$. In short, we have used    
$$ \sum_{\xi\in\Xi}  \sum_{\zeta\in \Omega_k}
     \left[
     |a_\zeta|^p
     \right]  =   \sum_{\zeta\in \Xi}    \sum_{\xi\in\Xi}    \left[ \chi_{\Omega_k(\xi)}(\zeta)
       \sum_{\zeta\in \Omega_k}|a_\zeta|^p
     \right] $$
in conjunction with the estimate $\#\{\xi\colon \ \zeta\in \Omega_k(\xi)\}\le K (2^{k+1} \Gamma)^d$ 
obtained from (\ref{pointcount}), since for  $\zeta\in \Xi$,  $\#\{\xi\colon \ \zeta\in \Omega_k(\xi)\} = \# \Omega_k(\zeta)$.
  In the final inequality, we have used the fact that $2^{(k+1)d} \le 2^{k(d+1)}\times 2^{d+1}$
 and that $\Gamma^{dp} = \frac{\Gamma^{(d+1)p}}{\Gamma^p}$. 
We estimate this with an integral as
$$
\sum_{\xi\in\Xi} I_{\xi}
\le
C
\left( 
  \Gamma^{dp}\exp(-\mu p \Gamma) +
  \frac{2}{\Gamma^{p}}
  \int_{\Gamma}^{\infty}  
    \exp\bigl(-\mu p r\bigr)  
    r^{(d+1)p-1} 
  \dif r 
\right)
  (\gamma q)^d 
  \rho^{p(m-\tfrac{d}2)} 
     \|\a\|_{\ell_p(\Xi)}^p.$$
 Which shows that
     $F(\Gamma) 
     := 
     C\left( 
  \Gamma^{dp}\exp(-\mu p \Gamma) +
  \frac{2}{\Gamma^{p}}
  \int_{\Gamma}^{\infty}  
    \exp\bigl(-\mu p r\bigr)  
    r^{(d+1)p-1} 
  \dif r 
\right)$. 

The integral term can be bounded by making a change of variable $R=r\Gamma$ as
$$\frac{2}{\Gamma^{p}}
  \int_{\Gamma}^{\infty}  
    \exp\bigl(-\mu p r\bigr)  
    r^{(d+1)p-1} 
  \dif r  
  = 
  2\Gamma^{dp}\int_1^{\infty}  \exp\bigl(-\mu p \Gamma R\bigr)  
    R^{(d+1)p-1} dR
    \le C_{d,p,m} \Gamma^{dp} \exp(-\mu p \Gamma). $$
Because  $\max_{\Gamma>1}\Gamma^{dp} \exp(-\frac{\mu}{2} p \Gamma) \le C_{d,p,m}$, the estimate   
$F(\Gamma)\le \tilde{C} e^{-\frac{\mu}{2} p \Gamma}$ follows. 
 \end{proof}

  \begin{lemma}\label{nearbypart}
  For every $p\in [1,\infty)$ and every $\Gamma>0$, we have the inequality
$$  \sum_{\xi\in \Xi} 
 \int_{B_{\xi}}
  \left|\sum_{\{\zeta\in \Xi :\ \zeta\neq \xi, |\zeta-\xi|\le \Gamma q\}}  a_{\zeta}\chi_{\zeta}(x)\right|^p
  \dif x
\le
C (\Gamma^{d} \gamma^{\epsilon})^p (\gamma q)^d   \rho^{p(m-\tfrac{d}2)}  \|\a\|_{\ell_p(\Xi)}^p
$$
  \end{lemma}
  \begin{proof}
  Note that $\#\{\zeta\in \Xi :\ \zeta\neq \xi, |\zeta-\xi|\le \Gamma q\} \le K \Gamma^d$, so by the quasi-triangle inequality, we have, for each $\xi \in \Xi$
\begin{align*}
II_{\xi}
  &\le \int_{B_{\xi}}(K \Gamma^d)^{p-1}\sum_{\d(\zeta,\xi)\le \Gamma q}|a_{\zeta} \chi_{\zeta}(x)|^p \dif x\\ 
&\le 
\int_{B_\xi} 
  (K \Gamma^d )^{p-1} (C \rho^{m-d/2}\gamma^{\epsilon})^p  
  \sum_{\d(\zeta,\xi) \le \Gamma q} 
    |a_\zeta|^p 
\dif x \\
 &\le 
C 
 \Gamma^{d(p-1)} \gamma^{\epsilon p} \rho^{p(m-d/2)}
 (\gamma q)^d 
 \sum_{\d(\zeta,\xi)\le \Gamma q}   |a_\zeta|^p.
%\label{ONE}
 \end{align*}
In the first inequality we use the  estimate on the number of centers (\ref{pointcount}).
In the second inequality, we use the bound (\ref{Holder_zero}).
The third inequality follows from the simple estimate $\mathrm{vol}(B_\xi)\le C (\gamma q)^d$.

Summing over $\xi \in\Xi$, we obtain:
\[
\sum_{\xi\in\Xi} II_{\xi}
\le
C   \Gamma^{d(p-1)}\gamma^{\epsilon p}(\gamma q)^d 
\sum_{\xi\in\Xi} \sum_{\d(\zeta,\xi)\le \Gamma q} |s(\zeta)|^p
\le C\left( \Gamma^{d}\gamma^{\epsilon} \right)^p(\gamma q)^d \rho^{p(m-\tfrac{d}2)} 
\|\a\|_{\ell_p(\Xi)}^p.
\]
The final estimate results by exchanging the two summations, and employing the fact that $\#\{\xi \in \Xi\colon \ \d(\zeta,\xi)\le \Gamma q\} \le K \Gamma^d$.
This completes the proof of the lemma.
\end{proof}

%%%%%%%%%%%%%%%%%%%%%%%%%%
%%%%%%%%%%%%%%%%%%%%%%%%%%
\begin{lemma}\label{offcenter}
There exists a constant  
$\gamma$ satisfying $\gamma \ge C \rho^{\frac{d-2m}{2\epsilon}} (\log(\rho))^{-d/\epsilon}$
with $C(d,m,p,\epsilon)$,    
so that
$$
2^{p-1}\sum_{\xi\in\Xi} 
  \int_{B_\xi}
      \Big|\sum_{\zeta\neq \xi} a_\zeta \chi_{\zeta}(x)\Big|^p 
   \dif x
\le
\frac{1}{2}\alpha (\gamma q)^d \left( \frac23\right)^p \sum_{\xi\in \Xi} |a_{\xi}|^p 
$$
holds for all $\a\in \ell_p(\Xi)$ and all $p\in [1,\infty)$.
\end{lemma}
%%%%%%%%%%%%%%%%%%%%%%%%%
%%%%%%%%%%%%%%%%%%%%%%%%%
%%%%%%%%%%%%%%%%%%%%%%%%%
\begin{proof}
By the quasi-triangle inequality, we have
$$\sum_{\xi\in\Xi}  \int_{B_\xi}\Big|\sum_{\zeta\neq \xi} a_\zeta \chi_{\zeta}(x)\Big|^p \dif x 
\le 
2^{p-1}
\left( 
\sum_{\xi\in\Xi}I_{\xi}
  + 
\sum_{\xi\in\Xi}II_{\xi}
\right) .$$
 Apply Lemma \ref{farawaypart}, and  choose 
 $\Gamma$  so that 
 $
 \tilde{C} e^{-\frac{\mu}{2}p \Gamma}
= \frac{1}{4}\alpha
    \left( \frac13\right)^p \rho^{-p(m-\tfrac{d}2)}  $, where $\tilde{C}$ is the constant appearing in Lemma \ref{farawaypart}.
 We note that our choice of $\Gamma$ guarantees $\Gamma \le  C_{d,p,m}  \log(\rho)$.
By Lemma \ref{farawaypart}, 
$F(\Gamma) 
\le  
\frac{1}{4}\alpha    \left( \frac13\right)^p \rho^{-p(m-\tfrac{d}2)}  
$,
    it then follows that
    $$ \sum_{\xi\in \Xi} 
 \int_{B_{\xi}}
  \left|\sum_{\{\zeta\in \Xi :\  |\zeta-\xi|\ge \Gamma q\}}  a_{\zeta}\chi_{\zeta}(x)\right|^p
  \dif x
   \le
   %\frac{4^{1-p}}{4}
   4^{-p}
   \alpha
    \left( \frac23\right)^p (\gamma q)^d \sum_{\xi\in\Xi}  |a_\xi|^p.$$

Now select $\gamma$ so that both $0<\gamma\le 1/(3 C\rho^{m-d/2})^{1/\epsilon}$ and  
$C(\Gamma^d \gamma^{\epsilon})^p \le  \frac{1}{4}\alpha
    \left( \frac23\right)^p \rho^{-p(m-\tfrac{d}2)}$ hold. 
     The problem of choosing  $\gamma$ can be rewritten as 
     $\gamma^{\epsilon} \le  \rho^{d/2-m} \min(\frac{1}{3C}, C_{d,m,p}\Gamma^{-d})$. Since $\Gamma^{-d}>  C_{d,p,m}  (\log(\rho))^{-d}$,
    it suffices to take $\gamma^{\epsilon} = C_{d,m,p} \rho^{d/2-m} (\log(\rho))^{-d}$ for some constant $C_{d,m,p}$.
    For this choice of $\gamma$,  Lemma \ref{nearbypart}  guarantees that 
$$ \sum_{\xi\in \Xi} 
 \int_{B_{\xi}}
  \left|\sum_{\{\zeta\in \Xi :\ \zeta\neq \xi, |\zeta-\xi|\le \Gamma q\}}  a_{\zeta}\chi_{\zeta}(x)\right|^p
  \dif x \le 
     %\frac{1}{4}
     4^{-p}
     \alpha
    \left( \frac23\right)^p (\gamma q)^d \sum_{\xi\in\Xi}  |a_\xi|^p$$
    as well.

    Thus, 
    $$
2^{p-1}\sum_{\xi\in\Xi} 
  \int_{B(\xi,\gamma q)}
      \Big|\sum_{\zeta\neq \xi} a_\zeta \chi_{\zeta}(x)\Big|^p 
   \dif x
   \le
 4^{p-1}\sum_{\xi\in\Xi} 
 ( I_{\xi}+II_{\xi} )
\le
\left(\frac{1}{4}+\frac{1}{4}\right)\alpha (\gamma q)^d \left( \frac23\right)^p \sum_{\xi\in \Xi} |a_{\xi}|^p
$$
and the result follows with $\gamma \ge C_{d,m,p,\epsilon} \rho^{\frac{d-2m}{2\epsilon}} (\log(\rho))^{-d/\epsilon}$.
    \end{proof}

%%%%%%%%%%%%%%%%%%%
%
%Lower bound
%
%%%%%%%%%%%%%%%%%%%
\begin{lemma}\label{lowercomparison}
Suppose $\Omega$ is a bounded domain and $\Xi\subset \Omega$ is a finite subset
with fill distance $h\le h_0:= \min(h_2,h_3)$, where $h_0=h_0(m,d)$. 
There exists a constant then the family of functions
$(\chi_{\xi})_{\xi\in \Xi}$ have the property that
for any $s = \sum_{\xi\in\Xi} a_{\xi} \chi_{\xi}$,
$$
c_1 q^{d/p}\| \bfa  \|_{\ell_p(\Xi)}       
\le  
\|s\|_p.
$$
holds with $c_1 \ge C \rho^{\frac{d(d-2m)}{2\epsilon p}} (\log(\rho))^{-\frac{d^2}{p\epsilon}}$ with   $0<\epsilon <m -d/2$ and
$C=C(d,p,m,\epsilon)$.
\end{lemma}
%%%%%%%%%%%%%%%%%%%%%%%%%
%%%%%%%%%%%%%%%%%%%%%%%%%
\begin{proof}
Since $s(\xi) = a_{\xi}$, the $L_{\infty}$ case follows immediately with constant $1$. For  $1 \le p < \infty$
we use (\ref{overshot}) and Lemma \ref{offcenter} to make the estimate
\begin{eqnarray*}
\alpha (\gamma q)^d \left( \frac23\right)^p \sum_{\xi\in \Xi} |a_{\xi}|^p 
&\le&
2^{p-1} 
\sum_{\xi\in \Xi}
  \int_{B_{\xi}}  
    \left(\left|\sum_{\zeta\in \Xi}  a_{\zeta}\chi_{\zeta}(x)\right|^p +  \left|\sum_{\zeta\neq \xi}  a_{\zeta}\chi_{\zeta}(x)\right|^p\right)
  \dif x.\\
&  \le&
\left(  2^{p-1} 
\sum_{\xi\in \Xi}
  \int_{B_{\xi}}  
 \left|\sum_{\zeta\in \Xi}  a_{\zeta}\chi_{\zeta}(x)\right|^p\right) \dif x+
 \frac{1}{2}\alpha (\gamma q)^d \left( \frac23\right)^p \sum_{\xi\in \Xi} |a_{\xi}|^p .
\end{eqnarray*}
Applying
$
\sum_{\xi\in \Xi}
  \int_{B_{\xi}}  
 \left|\sum_{\zeta\in \Xi}  a_{\zeta}\chi_{\zeta}(x)\right|^p 
 \le 
 \int_{\Omega}  \left|\sum_{\zeta\in \Xi}  a_{\zeta}\chi_{\zeta}(x)\right|^p \dif x$,
 the result follows with 
 $$c_1 = \frac13 (\alpha\gamma^d)^{1/p}  \ge 
 C_{d,m,p,\epsilon} \rho^{\frac{d(d-2m)}{2\epsilon p}} (\log(\rho))^{-\frac{d^2}{p\epsilon}}.$$

\end{proof}

\section*{Acknowledgement}
The authors are grateful to the anonymous referees for their helpful feedback which has improved the paper.
Additionally, the authors would like to thank Manolis Georgoulis for pointing out the trace estimates in section \ref{SS_trace}.

Christian Rieger acknowledges support of the Deutsche Forschungsgemeinschaft (DFG) through the Sonderforschungsbereich 1060: The Mathematics of Emergent Effects.

Some of the research was carried out during visits to the  Institute for Numerical Simulation at the University of Bonn.
Dr. Hangelbroek thanks the SFB 1060  for helping to support his visits.

\bibliographystyle{siam}
\bibliography{Bernstein_Bnd_8_21}
\end{document}